\documentclass[11pt]{amsart}

\usepackage{a4wide, amsmath, amsfonts, amssymb, mathrsfs, amsthm, breqn}
\usepackage[hyperfootnotes=false]{hyperref}
\usepackage[dvipsnames]{xcolor}
\usepackage{shuffle}
\usepackage{enumitem}
\allowdisplaybreaks[2]
\numberwithin{equation}{section}

\newtheorem{theorem}{Theorem}[section]
\newtheorem{lemma}[theorem]{Lemma}
\newtheorem{corollary}[theorem]{Corollary}
\newtheorem{remark}[theorem]{Remark}

\newtheorem{proposition}[theorem]{Proposition}

\newtheorem{example}[theorem]{Example}

\newcommand{\dd}{\,\mathrm{d}}
\renewcommand{\d}{\mathrm{d}}

\renewcommand{\epsilon}{\varepsilon}
\renewcommand{\phi}{\varphi}
\newcommand{\R}{\mathbb{R}}
\newcommand{\N}{\mathbb{N}}
\newcommand{\E}{\mathbb{E}}
\newcommand{\X}{\mathbb{X}}

\renewcommand{\P}{\mathbb{P}}

\newcommand{\bX}{\mathbf{X}}
\newcommand{\bY}{\mathbf{Y}}

\newcommand{\cA}{\mathcal{A}}
\newcommand{\cB}{\mathcal{B}}

\newcommand{\cF}{\mathcal{F}}
\newcommand{\cL}{\mathcal{L}}

\newcommand{\hX}{\widehat{X}}
\newcommand{\hY}{\widehat{Y}}
\newcommand{\hW}{\widehat{W}}
\newcommand{\hbX}{\widehat{\mathbf{X}}}

\newcommand{\hbW}{\widehat{\mathbf{W}}}

\newcommand{\bW}{\mathbf{W}}

\newcommand{\bbX}{\mathbb{X}}
\newcommand{\hbbX}{\widehat{\mathbb{X}}}
\newcommand{\hbbY}{\widehat{\mathbb{Y}}}
\newcommand{\hbbW}{\widehat{\mathbb{W}}}
\newcommand{\ver}[1]{{\vert\kern-0.25ex\vert\kern-0.25ex\vert #1 \vert\kern-0.25ex\vert\kern-0.25ex\vert}}

\DeclareMathOperator{\spn}{span}

\title[Global universality via discrete-time signatures]{Global universality via discrete-time signatures}

\author[Ceylan]{Mihriban Ceylan}
\address{Mihriban Ceylan, University of Mannheim, Germany}
\email{mihriban.ceylan@uni-mannheim.de}

\author[Pr{\"o}mel]{David J.~Pr{\"o}mel}
\address{David J. Pr{\"o}mel, University of Mannheim, Germany}
\email{proemel@uni-mannheim.de}

\date{\today}

\begin{document}

\begin{abstract}
  We establish global universal approximation theorems for non-anticipative and general path-dependent functionals on spaces of piecewise linear paths, stating that linear functionals of the corresponding signatures are dense with respect to $L^p$- and weighted norms. We verify that these approximation results are applicable to piecewise linear interpolations of a class of Gaussian processes, including  Brownian motion and fractional Brownian motion with Hurst parameter $H>1/4$. Moreover, we derive quantitative convergence rates for signatures of piecewise linear approximations of Gaussian processes towards their continuous-time counterparts. Consequently, we obtain $L^p$-approximation results for path-dependent functionals of Gaussian processes, as well as for random ordinary differential equations and stochastic differential equations driven by Brownian motion.
\end{abstract}

\maketitle




\section{Introduction}

Approximating functionals that depend in a complex way on the entire trajectory of a data stream is a fundamental challenge across a wide range of disciplines, including machine learning, quantitative finance, and the analysis of random dynamical systems. Many quantities arising in applications, such as optimal control strategies, option prices, and predictors learned from streamed data, depend on the full history of a signal. Developing tractable and expressive representations for such path-dependent objects is therefore of central importance, both from a theoretical and practical perspective.

Over the past decade, the signature of a path has emerged as a powerful and flexible framework for addressing such problems. Originally introduced by K.-T. Chen \cite{Chen1954} and later developed extensively within the rough path theory initiated by Lyons \cite{Lyons1998,Lyons2007}, the signature provides a systematic representation of a path in terms of its iterated integrals. Roughly speaking, the signature of a path is defined as the infinite collection of all its iterated integrals. This collection of features captures interactions between the components of the path across time and uniquely characterizes the path up to tree-like equivalence; see~\cite{Hambly2010,Boedihardjo2016}.  Owing to its rich algebraic structure, the signature possesses a fundamental universality property: linear combinations of signature coordinates are dense in spaces of continuous path functionals on compact subsets of path space; see, for example, \cite{Levin2013,Kidger2019,Lyons2019}. In this sense, signatures play a role analogous to classical polynomials: they form a universal feature class for continuous path-dependent functionals. Thanks to this universality and their favorable algebraic properties, signature-based methods have found numerous applications. They are widely used in machine learning for the analysis of sequential data, such as image and texture classification, the generation of synthetic data and topological data analysis, as well as in mathematical finance for tasks such as derivative pricing, model calibration, and stochastic control; see, for instance, \cite{Graham2013,Kidger2019,Lyons2019,Chevyrev2020,McLeod2025,Bank2025,Bayer2025a}.

Despite these strong theoretical guarantees, two fundamental challenges arise when applying signature-based methods in practice.

First, in most relevant settings the signature cannot be computed exactly. Even for classical stochastic processes such as Brownian motion, iterated integrals, forming the signature, cannot be evaluated pathwise in closed form. Consequently, practical implementations must rely on suitable approximations. A natural and structurally consistent approach is to start from discrete-time observations of the underlying signal. In practice, one typically has access only to a finite collection of observations rather than to the full continuous trajectory. To reconstruct a continuous object from such data, a widely used approach is to employ piecewise linear interpolation. This yields a continuous path that preserves the temporal ordering and increment structure of the observed signal while remaining computationally tractable. In particular, dedicated implementations such as the \verb|iisignature| library demonstrate that signatures and log-signatures of discretely sampled paths can be computed at relatively high truncation levels and in moderate dimensions with practical runtimes \cite{Reizenstein2018}. Thus, signatures of piecewise linearly interpolated observations are not merely a theoretical surrogate: they are among the objects that are actually computed in applications.

This intrinsic discreteness has led to increasing interest in signature-based methods for discrete-time systems and discretely observed paths. In reservoir computing, discrete-time signature-based systems have been used to explain universality phenomena in recurrent neural networks \cite{Cuchiero2021}. In mathematical finance, discrete-time variants of the signature have been employed to extract features directly from time series, to price exotic derivatives, and to address stochastic control problems \cite{Adachi2021,Lyons2019,Bank2025}, while related discretized It\^o-signature coordinates have recently been studied in the context of self-financing trading strategies and dynamic hedging \cite{Guo2026}. More broadly, truncated signatures of discretely sampled paths yield finite-dimensional representations of sequential data and have been successfully used as feature maps in a variety of contexts \cite{Fermanian2021}.
We emphasize, however, that different notions of discrete signatures coexist in the literature. In particular, the term `discrete-time signature` is also used for the iterated-sums signature of a time series, which is defined directly in terms of iterated sums of its increments and is naturally governed by the quasi-shuffle algebra \cite{Diehl2020}. In contrast, the present paper considers the ordinary iterated-integral signature of the piecewise linear interpolation of discrete observations. Thus, our construction yields a geometric signature with the usual shuffle structure and should not be confused with the iterated-sums signature.

Second, classical universal approximation theorems for signatures are typically formulated for continuous functionals on compact subsets of path space. While mathematically natural, this restriction can be limiting in many relevant settings. Indeed, many probabilistic models are not supported on compact subsets of natural path topology, for example, sample paths of processes such as Brownian motion do not lie in any fixed compact subset with positive probability. The same phenomenon persists for discretely observed paths and their piecewise linear interpolations, whose trajectories are not naturally confined to compact sets.

These observations motivate the development of global approximation results for signatures of discretely observed paths. In particular, it is natural to study approximation in weighted function spaces and in $L^p$-spaces, which provide a flexible framework for controlling the growth of functionals outside compact sets. Establishing such global universal approximation results extends the classical compact-set theory and better reflects the requirements arising in probabilistic modeling, machine learning, and quantitative finance.

In this paper, we establish global universal approximation theorems for both general path-dependent and non-anticipative functionals on spaces of piecewise linear paths, as they arise from discrete-time observations and in numerical approximation. We first prove density results in weighted function spaces of (stopped) paths (Proposition~\ref{prop:UAT weighted} and Proposition~\ref{prop:UAT non-anticipative}). The weighted framework is inspired by and adapted from recent developments in \cite{Cuchiero2024}. In particular, our density results rely on the weighted Stone--Weierstra\ss{} theorem established therein, which we extend to the setting of piecewise linear paths. Subsequently, we derive $L^p$-universal approximation theorems (Theorem~\ref{thm:Lpmain}), stating that linear functionals acting on the signatures of time-extended paths are dense with respect to the standard $L^p$-metric, provided the underlying finite measure satisfies an exponential moment condition. Our $L^p$-approximation results are then formulated directly on path spaces generated by discrete-time data.

To position our contribution within the existing literature, we emphasize that, in contrast to density results in weighted function spaces and $L^p$-approximation results obtained in the continuous setting for geometric rough paths (see, e.g., \cite{Schell2023,Bayer2025,Cuchiero2024b,Ceylan2025,Chevyrev2026}), our framework is intrinsically discrete and built on piecewise linear interpolations of discretely observed signals. Among these, the framework developed in \cite{Ceylan2025} is particularly close to the present work. While the $L^p$-approximation result for a fixed partition can formally be transferred from the continuous-time framework of \cite{Ceylan2025}, the corresponding weighted approximation theorem does not follow directly, since weighted control on the discretely generated subspace does not provide weighted control on the entire rough path space. Consequently, the ambient weighted approximation theorem cannot simply be restricted to the discrete path space. A separate analysis of the piecewise linear path and stopped-path spaces, their topological properties, and the compactness of the weight sublevel sets is therefore required. We refer to Remark~\ref{rem: other Lp results} for a more detailed discussion.

The strength of the discrete-time global framework becomes particularly apparent in probabilistic applications. We first show that piecewise linear interpolations of a suitable class of Gaussian processes satisfy the exponential integrability condition required by our $L^p$-approximation theorems. This yields approximation results for $p$-integrable general path-dependent and non-anticipative functionals of the interpolated Gaussian process, see Subsection~\ref{subsec: Gaussian processes}. Standard Brownian motion and fractional Brownian motion with Hurst parameter $H>1/4$ provide prominent examples of the suitable class of Gaussian processes. Notably, the present result covers the broader range $H>1/4$ for signatures of piecewise linear interpolations, compared with $H>1/3$ for the canonical continuous rough path signatures treated in \cite{Ceylan2025}. We then study the error introduced by replacing the continuous Gaussian signature with the signature of a piecewise linear interpolation. Under suitable covariance regularity assumptions, we obtain an explicit algebraic convergence rate in terms of the mesh size, see Subsection~\ref{subsec: approximation error}. In the Brownian setting, the rate improves to the sharp convergence rate for every fixed linear functional of the signature. These estimates provide a quantitative link between continuous-time signatures and the discrete signatures that can be computed from observations.

Finally, the Brownian case allows us to connect the discrete framework with stochastic differential equations. Recent continuous-time results such as \cite{Ceylan2025} show that solutions to stochastic differential equations driven by Brownian motion, as well as functionals on the Brownian rough path, can be approximated in $L^p$-spaces by linear functionals of the Brownian signature. In the present discrete-time setting, we show that these linear functionals acting on the signature of the piecewise linearly interpolated sample paths of Brownian motion can approximate the solutions to random ordinary as well as stochastic differential equations driven by Brownian motion, see Section~\ref{sec: approximation of SDEs}.

\medskip
\noindent\textbf{Organization of the paper:} In Section~\ref{sec: preliminaries} we recall the essential concepts regarding signatures and rough path theory. In Section~\ref{sec: global UAT}, the universal approximation theorems on spaces of piecewise linear paths are established, stating that linear functionals of the corresponding signatures are dense with respect to $L^p$-and weighted norms, under an integrability condition on the underlying weight function. In Section~\ref{sec: applications}, we verify this integrability condition for Gaussian processes, derive approximation results for Gaussian path functionals, and establish convergence rates for signatures of piecewise linear approximations of Gaussian processes. Section~\ref{sec: approximation of SDEs} we deduce $L^p$-approximation results for random ordinary differential equations and stochastic differential equations driven by Brownian motion.

\medskip
\noindent\textbf{Acknowledgments:} M. Ceylan gratefully acknowledges financial support by the doctoral scholarship programme from the Avicenna-Studienwerk, Germany, and D.J. Pr\"omel his affiliation with the Department of Mathematics at King’s College London, United Kingdom.

\section{Preliminaries}\label{sec: preliminaries}

First we recall some essentials on signatures and rough path theory, where we refer to~\cite{Friz2010, Friz2020} for a more detailed introduction.

\subsection{Algebraic setting for signatures}

For $d\in\N$, let $\R^d$ be the standard $d$-dimensional Euclidean space equipped with the norm $|x|_{\R^d}:=(\sum_{i=1}^d x_i^2)^{1/2}$ for $x=(x_1,\dots,x_d)\in\R^d$. When the underlying space is clear from context, we simply write $|x|$ instead of $|x|_{\R^d}$.

The n-fold tensor product of $ \R^d$ is given by
\begin{equation*}
  (\R^d)^{\otimes 0}:=\R \quad \text{and}\quad(\R^d)^{\otimes n}:=\underbrace{\R^d\otimes\cdots\otimes\R^d}_{n},\quad \text{for }n\in \N.
\end{equation*}
Let $(e_1, \ldots, e_d)$ be the canonical basis of $\R^d$. It is well-known that $\{e_{i_1}\otimes\cdots\otimes e_{i_n}: i_1,\ldots,i_n\in \{1,\ldots,d\}\}$ is a canonical basis for $(\R^d)^{\otimes n}$ and we denote by $e_\emptyset$ the basis element of $(\R^d)^{\otimes 0}$.
Then, every $a^{(n)}\in(\R^d)^{\otimes n}$ admits the coordinate representation
\begin{equation*}
  a^{(n)}=\sum_{i_1,\dots,i_n=1}^d a_{i_1,\dots,i_n}\, e_{i_1}\otimes\cdots\otimes e_{i_n},
\end{equation*}
and we equip $(\R^d)^{\otimes n} $ with the usual Euclidean norm
\begin{equation*}
  | a^{(n)}|_{(\R^d)^{\otimes n}}:=\bigg(\sum_{i_1,\ldots,i_n=1}^d|a_{i_1,\ldots,i_n}|^2\bigg)^{1/2},\quad\text{for }  a^{(n)}\in (\R^d)^{\otimes n}.
\end{equation*}
When no confusion may arise, we write $| a^{(n)} |$ instead of $| a^{(n)}|_{(\R^d)^{\otimes n}}$.

For $d \in \mathbb{N}$, the extended tensor algebra on $\R^{d}$ is defined as
\begin{equation*}
  T((\R^{d})):=\Bigl\{\mathbf{a}:=(a^{(0)}, \ldots,  a^{(n)}, \ldots): a^{(n)} \in(\R^{d})^{\otimes n}\Bigr\},
\end{equation*}
and $a^{(i)}$ is called tensor of level $i$. We equip $T((\R^d))$ with the standard addition ``$+$'', tensor multiplication ``$\otimes$'', and scalar multiplication ``$\cdot$'' defined by
\begin{align*}
  \mathbf{a}+\mathbf{b} & :=\Bigl( a^{(0)}+ b^{(0)}, \ldots,  a^{(n)}+ b^{(n)}, \ldots\Bigr), \\
  \mathbf{a} \otimes \mathbf{b} & :=\Bigl( c^{(0)}, \ldots,  c^{(n)}, \ldots\Bigr),\\
  \lambda\cdot \mathbf{a} & :=\Bigl(\lambda  a^{(0)}, \ldots, \lambda a^{(n)}, \ldots\Bigr),
\end{align*}
for $\mathbf{a}=(a^{(n)})_{n=0}^{\infty}, \mathbf{b} =(b^{(n)})_{n=0}^{\infty}\in T((\R^{d}))$ and $\lambda \in \R$, where $ c^{(n)}:=\sum_{k=0}^{n} a^{(k)} \otimes  b^{(n-k)}$. Let us remark that $(T((\mathbb{R}^{d})),+, \cdot, \otimes)$ is a real non-commutative algebra with neutral element $\mathbf{1}=(1,0, \ldots, 0, \ldots)$. Similarly, we define the truncated tensor algebra of order $N \in \N$ by
\begin{equation*}
  T^{N}(\R^{d}):=\Bigl\{\mathbf{a} \in T((\R^{d})):  a^{(n)}=0, \forall n>N\Bigr\},
\end{equation*}
which we equip with the norm 
\begin{equation*}
  \|\mathbf a\|_{T^N(\R^d)}:=\max_{n=0,\ldots, N}| a^{(n)}|_{(\R^d)^{\otimes n}},\quad \text{for } \mathbf a=(a^{(n)})_{n=0}^N \in T^N(\R^d).
\end{equation*}
Note that $T^{N}(\R^{d})$ has dimension $\sum_{i=0}^{N} d^{i}=$ $(d^{N+1}-1) /(d-1)$. Additionally, we define the tensor algebra $T(\R^d)=\bigcup_{n\in\N}T^n(\R^d)$ and consider the truncated tensor subalgebras $T_0^N(\R^d), T_1^N(\R^d)\subset T^N(\R^d)$ of elements $\mathbf a\in T^N(\R^d)$ with $a^{(0)}=0, a^{(0)}=1$, respectively. Observe that $T_1^N(\R^d)$ is a Lie group under $\otimes$, with unit element $\mathbf 1=(1,0,\ldots,0)$.

The Lie algebra that is generated from $\{\mathbf{e}_1, \dots, \mathbf{e}_d\}$, where $\mathbf{e}_i := (0,e_i,0,\dots) \in T(\R^d)$, and the commutator bracket
\begin{equation*}
  [\mathbf{a},\mathbf{b}] = \mathbf{a} \otimes \mathbf{b} - \mathbf{b} \otimes \mathbf{a}, \qquad \mathbf{a}, \mathbf{b} \in T(\R^d),
\end{equation*}
is called the free Lie algebra $\mathfrak{g}(\R^d)$ over $\R^d$, see e.g. \cite[Section~7.3]{Friz2010}. It is a subalgebra of $T_0((\R^d))$, where we define for $c \in \R$, the tensor subalgebra $T_c((\R^d)) := \{\mathbf{a} = (a^{(n)})_{n=0}^\infty \in T((\R^d)): a^{(0)} = c\}$. The free Lie group $G((\R^d)) := \exp(\mathfrak{g}(\R^d))$ is defined as the tensor exponential of $\mathfrak{g}(\R^d)$, i.e., the image of $\mathfrak{g}(\R^d)$ under the map
\begin{equation*}
  \exp_{\otimes} \colon T_0((\R^d)) \to T((\R^d)), \qquad \mathbf{a} \mapsto 1 + \sum_{k=1}^{\infty} \frac{1}{k!} \mathbf{a}^{\otimes k}.
\end{equation*}
$G((\R^d))$ is a subgroup of $T_1((\R^d))$. In fact, $(G((\R^d)),\otimes)$ is a group with unit element $(1,0,\dots,0,\dots)$, and for all $\mathbf{g} = \exp_{\otimes}(\mathbf{a}) \in G((\R^d))$, the inverse with respect to $\otimes$ is given by $\mathbf{g}^{-1} = \exp_{\otimes}(-\mathbf{a})$, for $\mathbf{g} = \exp_{\otimes}(\mathbf{a}) \in G((\R^d))$. We call elements in $G((\R^d))$ group-like elements. For $N\in \N$, we define the free step-$N$ nilpotent Lie algebra $\mathbf {\mathfrak g}^N(\R^d)\subset T_0^N(\R^d)$ with
\begin{equation*}
  \mathbf {\mathfrak g}^N(\R^d):=\{0\}\oplus\R^d\oplus[\R^d,\R^d]\oplus\cdots\oplus\underbrace{[\R^d,[\ldots,[\R^d,\R^d]]]}_{(N-1)\text{ brackets}},
\end{equation*}
where $(\mathbf g,\mathbf h)\mapsto [\mathbf g,\mathbf h]:=\mathbf g\otimes\mathbf h-\mathbf h\otimes\mathbf g\in T_0^N(\R^d)$ denotes the Lie bracket for $\mathbf g, \mathbf h\in T^N(\R^d)$, see \cite[Chapter~7.3.2 and Definition~7.25]{Friz2010}. The image $G^N(\R^d):=\exp(\mathbf{\mathfrak g}^N(\R^d))$ is a (closed) sub-Lie group of $(T_1^N(\R^d),\otimes)$, called the free nilpotent group of step $N$ over $\R^d$, see \cite[Theorem~7.30]{Friz2010}.

We define $I:=(i_1,\ldots,i_n)$ as a $n$-dimensional multi-index of non-negative integers, i.e. $i_j\in\{1,\ldots,d\}$ for every $j\in\{1,2,\ldots,n\}$. Note that $|I|:=n$ and the empty index is given by $I:=\emptyset$ with $|I|=0$. For $n\ge 1$ or $n\ge2$, we write $I^\prime:=(i_1,\ldots,i_{n-1})$ and $I^{\prime\prime}:=(i_1,\ldots,i_{n-2})$, respectively. Moreover, for each $|I|\ge 1$, we set $e_I:=e_{i_1}\otimes\cdots\otimes e_{i_n}$. This allows us to write $\mathbf{a} \in T((\R^d))$ (and $\mathbf{a} \in T(\R^d)$) as
\begin{equation*}
  \mathbf{a} = \sum_{|I| \geq 0} \langle e_I,\mathbf a\rangle e_I,
\end{equation*}
where $\langle \cdot,\cdot\rangle$ is defined as the inner product of $(\R^d)^{\otimes n}$ for each $n\ge 0$. 

For two multi-indices $I = (i_1, \ldots, i_{|I|})$ and $J = (j_1, \ldots, j_{|J|})$ with entries in $\{1,\ldots,d\}$, the shuffle product is recursively defined by
\begin{equation*}
  e_I \shuffle e_J := (e_{I'} \shuffle e_J) \otimes e_{i_{|I|}} + (e_I \shuffle e_{J'}) \otimes e_{j_{|J|}},
\end{equation*}
with $e_I \shuffle e_\emptyset := e_\emptyset \shuffle e_I := e_I$. For all $\mathbf{a} \in G((\R^d))$, the shuffle product property holds, i.e., for two multi-indices $I = (i_1, \ldots, i_{|I|})$ and $J = (j_1, \ldots, j_{|J|})$, it holds that
\begin{equation*}
  \langle e_I, \mathbf{a} \rangle \langle e_J, \mathbf{a} \rangle = \langle e_I \shuffle e_J, \mathbf{a} \rangle.
\end{equation*}

\subsection{Essentials on rough path theory}

The space of continuous linear maps $f$ from the normed space $(X,\|\cdot\|_X)$ to the normed space $(Y,\|\cdot\|_Y)$ is denoted by $\cL(X;Y)$, which is equipped with the norm $\|f\|_{\cL(X;Y)}:=\sup_{x\in X,\|x\|_X\le 1}\|f(x)\|_Y$. Furthermore, if $Y=\R$, the topological dual space of $X$, denoted by $X^\ast$, is identified with $\cL(X;\R)$. Elements of $X^\ast$ are linear functionals $\ell\colon X\to\R$ and the norm on $X^\ast$ is defined by $\|\ell\|_{X^\ast}:=\sup_{x\in X,\|x\|_X\le 1}|\ell(x)|$.

\medskip

For a Hausdorff topological space $(X,\tau_X)$ and a normed space $(E,\|\cdot\|_E)$, the space of continuous functions $f\colon X\to E$ is denoted by $C(X;E)$ and $C_b(X;E)\subseteq C(X;E)$ denotes the vector subspace of bounded functions. Whenever $E=\R$, we simplify the notation to $C(X):=C(X;\R)$ and $C_b(X):=C_b(X;\R)$, respectively. We write $C_b^k=C_b^k(\R^m;\cL(\R^d;\R^m))$ for the space of $k$-times continuously differentiable functions $f\colon \R^m\to\cL(\R^d;\R^m)$ such that $f$ and all its derivatives up to order $k$ are continuous and bounded, and equip the space $C_b^k=C_b^k(\R^m;\cL(\R^d;\R^m))$ with the norm
\begin{equation*}
  \|f\|_{C_b^k}:=\|f\|_{\infty}+\|Df\|_{\infty}+\ldots+\|D^kf\|_{\infty},
\end{equation*}
where $D^r f$ denotes the $r$-th order derivative of $f$ and $\|\cdot\|_\infty$ denotes the supremum norm on the corresponding spaces of operators.

For a measure space $(X,\mathcal A,\mu)$ and $1\le p<\infty$, the (vector-valued) Lebesgue space $L^p(X,\mu;\R^d)$ is defined as the space of (equivalence classes of) $\mathcal A$-measurable functions $f\colon X\to\R^d$ such that
\begin{equation*}
  \|f\|_{L^p(X,\mu;\R^d)} := \Bigl(\int_X |f(x)|^p \,\dd\mu(x)\Bigr)^{\frac{1}{p}} < \infty.
\end{equation*} 
For $d=1$, we simply write $L^p(X):=L^p(X,\mu):= L^p(X,\mu;\R)$ and $\|\cdot\|_{L^p(X)}:=\|\cdot\|_{L^p(X,\mu;\R)}$.

\medskip

Let $T>0$ be a fixed finite time horizon, then we define the $1$-variation of a path $X\colon [0,T]\to\R^d$ by
\begin{equation*}
  \|X\|_{1\textup{-var}}:=\sup_{\mathcal D\subset [0,T]}\Bigl(\sum_{t_i\in\mathcal D}|X_{t_i}-X_{t_{i-1}}|\Bigr),
\end{equation*}
where the supremum is taken over all partitions $\mathcal D=\{0=t_0<t_1<\cdots<t_n=T\}$ of the interval $[0,T]$ and $\sum_{t_i\in\mathcal D}$ denotes the summation over all points in $\mathcal D$. Then, if $\|X\|_{1\textup{-var}}<\infty$, we say that $X$ is of bounded variation or of finite $1$-variation on $[0,T]$. The space of continuous paths of bounded variation on $[0,T]$ with values in $\R^d$ is denoted by $C^{1\textup{-var}}([0,T];\R^d)$.

\medskip

Let $\pi=\{t_i\}_{i=0}^{n}=\{0=t_0<\cdots<t_{n}=T\}$ be a partition of the interval $[0,T]$. Further, fix an initial value $x_0\in\R^d$, and let $X\colon \pi\to\R^d$ be such that $X_0=x_0$. The continuous bounded variation path $X^\pi\colon [0,T]\to\R^d$ is given by $\{X_{t_k}\}_{k=0}^{n}$ and a linear interpolation in between:
\begin{equation*}
  X^\pi_t=X_{t_{k}}+\frac{t-t_k}{t_{k+1}-t_k}(X_{t_{k+1}}-X_{t_k}),\quad t\in[t_k,t_{k+1}],\quad k=0,\ldots,n-1.
\end{equation*}
The time-extended continuous path $\hX^\pi\colon[0,T]\to\R^{d+1}$ with bounded variation is given by
\begin{equation*}
  \hX_{t_k}:=(t_k,X_{t_k}),
\end{equation*}
and for the values in between we again use a linear interpolation
\begin{equation*}
  \hX^\pi_t=\hX_{t_{k}}+\frac{t-t_k}{t_{k+1}-t_k}(\hX_{t_{k+1}}-\hX_{t_k}),\quad t\in[t_k,t_{k+1}],\quad k=0,\ldots,n-1.
\end{equation*}
We denote by $\mathcal C^\pi:=\{X^\pi:X\colon\pi\to\R^d ,X_0=x_0\}$ the space of bounded variation paths with fixed initial value $x_0\in \R^d$ constructed by a linear interpolation and by
\begin{equation*}
  \widehat{\mathcal C}^\pi:=\{\hX^\pi: X\colon\pi\to\R^d, X_0=x_0\}
\end{equation*}
the corresponding space of time-extended paths.

Observe that the spaces $\mathcal C^\pi$ and $\widehat{\mathcal C}^\pi$ are subspaces of the space of paths with bounded variation.

\begin{remark}
  Throughout this work, we focus on piecewise linear approximations, as they provide a canonical and widely used approximation scheme in signature-based methods. We nevertheless note that the deterministic results presented below, together with their underlying arguments, may extend to other finite-dimensional interpolation schemes, provided that the corresponding path and stopped-path spaces are Polish and retain the compactness and signature-continuity properties required in our analysis, in particular those ensuring that they remain weighted spaces. We do not pursue such extensions here.
\end{remark}

We define the signature of a continuous path of bounded variation as follows: Let $X\in C^{1\textup{-var}}([0,T];\R^d)$ be a path of bounded variation. Then, we denote its signature truncated  at level $N$ on $[s,t]\subset[0,T]$ by
\begin{equation*}
  \X _{s,t}^N:=\Bigl(1,\int_{s<u<t}\dd X_u,\ldots, \int_{s<u_1<\ldots<u_N<t}\dd X_{u_1}\otimes\cdots\otimes\dd X_{u_N}\Bigr)\in T^N(\R^d),
\end{equation*}
for $0\le s\le t\le T$, where the integrals are defined via Riemann--Stieltjes integration, and
\begin{equation*}
  \X_{s,t} :=(1,\X_{s,t}^{(1)},\X_{s,t}^{(2)},\ldots,\X_{s,t}^{(N)},\ldots)\in T((\R^d)),
\end{equation*}
the signature, where
\begin{equation*}
  \X_{s,t}^{(n)}:=\int_{s<u_1<\ldots<u_n<t}\dd X_{u_1}\otimes\cdots\otimes\dd X_{u_n}
\end{equation*}
denotes the $n$-th component of $\X_{s,t}$. For $s=0$ we simply write $\X_t$.
  
Further, for a normed space $(E,\|\cdot\|_E)$ and $\alpha\in (0,1]$, the $\alpha$-H\"older norm of a path $X\in C([0,T];E)$ is given by
\begin{equation*}
  \|X\|_\alpha:=\sup_{0\le s<t\le T}\frac{\|X_t-X_s\|_E}{|t-s|^\alpha}.
\end{equation*}
We write $C^\alpha([0,T];E)$ for the space of paths  $X\in C([0,T];E)$ which satisfy $\|X\|_\alpha<\infty$ and call this space the space of $\alpha$-H\"older continuous paths.

The corresponding metric, we denote by
\begin{equation*}
  d_\alpha(X,Y):=\sup_{\overset{s,t\in[0,T]}{s<t}}\frac{\|X_{s,t}-Y_{s,t}\|_{E}}{|t-s|^\alpha},
\end{equation*}
where $X_{s,t}:=X_t-X_s$ for paths $X,Y\colon [0,T]\to E$, with the same initial point $X_0=Y_0=x_0,$ $x_0\in\R^d$. Moreover, we define the metric
\begin{equation*}
  d_{\infty}(X,Y):=\sup_{t\in[0,T]}\|X_t-Y_t\|_E,
\end{equation*}
for paths $X,Y\colon [0,T]\to E$, with $X_0=Y_0=x_0$.

\begin{remark}
  Observe that by constructing paths through linear interpolation, these paths are even Lipschitz continuous and thus $1$-H\"older continuous. This allows us to equip this space with the $1$-H\"older norm. Consequently, the spaces $\mathcal C^\pi$ and $\widehat{\mathcal C}^\pi$ can be regarded as subspaces of the space of $1$-H\"older continuous paths. Furthermore, as the space of  $1$-H\"older continuous paths is itself a subspace of the space of $\alpha$-H\"older continuous paths, $\mathcal C^\pi$ and $\widehat{\mathcal C}^\pi$ can be equipped with any $\alpha$-H\"older norm for any $\alpha\in(0,1]$. Finally, note that these spaces are also Polish, as all paths in $\mathcal C^\pi$ and $\widehat{\mathcal C}^\pi$ are piecewise linear.
\end{remark}

In the following we denote by $\X^\pi:=(1,X^\pi,\bbX^{\pi,(2)},\ldots)$ and $\hbbX^\pi:=(1,\hX^\pi,\hbbX^{\pi,(2)},\ldots)$ the signature of paths in $\mathcal C^\pi$ and $\widehat{\mathcal C}^\pi$, respectively, where the signature is defined as the tensor series of iterated (Riemann--Stieltjes) integrals. For $N \in \N$, the signature of $\hbbX^\pi$ truncated at level $N$ is the path $t \mapsto \hbbX^{\pi,N}_t$ defined by computing all iterated integrals up to order $N$.

\medskip

The Carnot--Carathéodory norm $\|\cdot\|_{cc}$ on $G^N(\R^d)$ is defined by
\begin{equation*}
  \|\mathbf g\|_{cc}:=\inf\biggl\{\int_{0}^{T}|\dd X_t|\,:X\in C^{1\textup{-var}}([0,T];\R^d)~\text{such that}~ \bbX_T^N=\mathbf g\biggr\},
\end{equation*}
for $\mathbf g\in G^N(\R^d)$. This norm induces a metric via
\begin{equation*}
  d_{cc}(\mathbf g, \mathbf h):=\|\mathbf{g}^{-1}\otimes \mathbf h\|_{cc},
\end{equation*}
for $\mathbf g, \mathbf h \in G^N(\R^d)$.

\medskip

For paths of lower regularity, i.e. for paths of finite $\alpha$-H\"older norm with $\alpha\in(0,1]$, it is necessary to work within the framework of rough path theory. In particular, we recall the definition of (weakly) geometric $\alpha$-H\"older rough paths.

For $\alpha\in(0,1]$, a continuous path $\bX\colon [0,T]\to G^{\lfloor 1/ \alpha\rfloor}(\R^d)$ of the form
\begin{equation*}
  [0,T]\ni t\mapsto \bX_t:=(1,X_t,\bbX_t^{(2)},\ldots,\bbX_t^{(\lfloor 1/ \alpha\rfloor)})\in G^{\lfloor 1/ \alpha\rfloor}(\R^d)
\end{equation*}
with $\bX_0:=\mathbf{1}:=(1,0,\ldots,0)\in G^{\lfloor 1/ \alpha\rfloor}(\R^d)$ is called geometric $\alpha$-H\"older rough path if the $\alpha$-H\"older rough path norm
\begin{equation*}
  \|\bX\|_{cc,\alpha}:=\sup_{\overset{s,t\in[0,T]}{s<t}}\frac{d_{cc}(\bX_s,\bX_t)}{|s-t|^\alpha}
\end{equation*}
is finite. We denote by $C_0^\alpha([0,T];G^{\lfloor 1/ \alpha\rfloor}(\R^d))$ the space of such geometric $\alpha$-H\"older rough paths, which we equip with the metric
\begin{equation*}
  d_{cc,\alpha}(\bX,\bY):=\sup_{\overset{s,t\in[0,T]}{s<t}}\frac{d_{cc}(\bX_{s,t},\bY_{s,t})}{|s-t|^{\alpha}},
\end{equation*}
for $\bX,\bY\in C_0^\alpha([0,T];G^{\lfloor 1/ \alpha\rfloor}(\R^d))$, where $\bX_{s,t}:=\bX_s^{-1}\otimes \bX_t\in G^{\lfloor 1/ \alpha\rfloor}(\R^d)$. Moreover, we define the metric
\begin{equation*}
  d_{cc,\infty}(\bX,\bY):=\sup_{t\in[0,T]}d_{cc}(\bX_t,\bY_t),
\end{equation*}
for $\bX,\bY\in C_0^\alpha([0,T];G^{\lfloor 1/ \alpha\rfloor}(\R^d)).$

\medskip

For $\alpha\in(0,1]$, let $\mathbf X\in C_0^\alpha([0,T];G^{\lfloor 1/ \alpha\rfloor}(\R^d))$ be a geometric $\alpha$-H\"older rough path. Then, the truncated signature of $\mathbf X$ is defined as the unique path extension of $\mathbf X$ yielding a path
\begin{equation*}
  \X^N\colon [0,T]\to G^N(\R^d),
\end{equation*}
for $N>\lfloor 1/ \alpha\rfloor$ specified by Lyons' extension theorem. Then, $\X^N$ has finite $\alpha$-H\"older norm $\|\cdot\|_{cc,\alpha}$ and starts with the unit element $\mathbf 1\in G^N(\R^d)$. The signature of $\mathbf X$ is given by
\begin{equation*}
  [0,T]\ni t\mapsto \X_t:=(1,\X_t^{(1)},\ldots,\X_t^{(\lfloor 1/ \alpha\rfloor)},\ldots,\X_t^{(N)},\ldots).
\end{equation*}

\medskip 

Note that for $\hX^\pi\in\widehat{\mathcal C}^\pi$,  $\hbX^\pi:=(1,\hX^\pi)$ is a geometric $1$-H\"older rough path, that is $\hbX_t^\pi=(1,\hX_t^\pi)\in G^1(\R^{d+1})$ and has finite $\alpha$-H\"older rough path norm for any $\alpha\in(0,1]$. Therefore, $\hbbX_t^{\pi, N}$ is the canonical lift of $\hbX_t^\pi$ to a path with values in $G^N(\R^{d+1})$ (due to the integration by parts rule), and is $\alpha$-H\"older continuous, i.e., $\hbbX^{\pi, N} \in C_0^\alpha([0,T];G^N(\R^{d+1}))$, for any $\alpha\in(0,1]$.

\section{Global universal approximation}\label{sec: global UAT}

In this section, we explore the approximation of $L^p$-functionals using linear combinations of signatures in discrete time. This approach is based on the concept of global universal approximation theorems (UAT) on weighted function spaces, as introduced in \cite{Cuchiero2024}. 

\subsection{Universal approximation on weighted spaces}

We start by introducing the concept of weighted spaces, which will be crucial for obtaining a global universal approximation result. Unlike the standard universal approximation theorem, see for instance \cite{Lyons2019,Kiraly2019} for the universality result on bounded variation spaces, this global version extends beyond the constraint of compact sets. We follow the weighted space framework of \cite{Cuchiero2024}, where the authors demonstrated that the space of weakly geometric rough paths forms a weighted space with an appropriate weight function. Similarly, we aim to show that the space $\widehat{\mathcal C}^\pi$ and the space of stopped paths, introduced later, are also weighted spaces, and in turn to deduce a universal approximation result on weighted spaces using the \emph{weighted Stone--Weierstraß theorem} introduced in \cite{Cuchiero2024}.

\medskip

To begin, we recall the framework of weighted function spaces. First, we define the weighted space, which then allows to define the corresponding weighted function space.

Let $(X,\tau_X)$ be a completely regular Hausdorff topological space. A function $\psi\colon X\to (0,\infty)$ is called an admissible weight function if every pre-image $K_R:=\psi^{-1} ((0,R])=\{x\in X: \psi(x)\le R\}$ is compact with respect to $\tau_X$, for all $R>0$. In this case, we call the pair $(X,\psi)$ a weighted space.

Furthermore, we define the vector space 
\begin{equation*}
  B_\psi(X):=\Bigl\{f\colon X\to \R: \sup_{x\in X}\frac{|f(x)|}{\psi(x)}<\infty\Bigr\},
\end{equation*}
consisting of functions $f\colon X\to \R$, whose growth is controlled by the growth of the weight function $\psi\colon X\to (0,\infty)$, which we equip with the weighted norm $\|\cdot\|_{\cB_\psi(X)}$ given by
\begin{equation}\label{eq:weightednorm}
  \|f\|_{\cB_\psi(X)}:=\sup_{x\in X}\frac{|f(x)|}{\psi(x)},\quad f\in B_\psi(X).
\end{equation}
Note that the embedding $C_b(X)\hookrightarrow B_\psi(X)$ is continuous, allowing us to introduce the space
\begin{equation*}
  \mathcal{B}_\psi(X)
  := \overline{C_b(X)}^{\,\|\cdot\|_{\mathcal{B}_\psi(X)}},
\end{equation*}
which is the closure of $C_b(X)$ with respect to the norm $\|\cdot\|_{\cB_\psi(X)}$. Note that $\mathcal{B}_\psi(X)$ is a Banach space with the norm~\eqref{eq:weightednorm}. We refer to $\cB_\psi(X)$ as a weighted function space.

\medskip

In the following let us discuss, why the space $\widehat{\mathcal C}^\pi$, equipped with the $\alpha^\prime$-H\"older norm for $0\le\alpha^\prime<\alpha\le 1$, is a weighted space. To this end, we define the weight function
\begin{equation}\label{eq: weight function}
  \psi\colon\widehat{\mathcal C}^\pi\to(0,\infty),\quad \psi(\hX^\pi):=\exp(\beta\|\hX^\pi\|_{\alpha}^\gamma),
\end{equation}
for some $\beta>0$, $\gamma\ge 1$ and $\alpha\in(0,1]$. This choice of the weight function is motivated by the weighted Stone--Weierstraß theorem, see \cite[Theorem~3.9]{Cuchiero2024}. As discussed before, we can equip the space $\widehat{\mathcal C}^\pi$ with any $\alpha$-H\"older norm. To ensure the space becomes a weighted space, we equip it with a weaker topology than the norm topology. In particular, we equip this space with the $\alpha^\prime$-H\"older norm for some $0\le \alpha^\prime<\alpha\le 1$. This is necessary to obtain an admissible weight function, i.e., to ensure that the closed unit ball
\begin{equation*}
  K_R:=\{\hX^\pi\in\widehat{\mathcal C}^\pi: \exp(\beta\|\hX^\pi\|_{\alpha}^\gamma)\le R\}
\end{equation*}
is compact w.r.t.~the weaker topology. As we can compactly embed $\alpha$-H\"older spaces into $\alpha^\prime$-H\"older spaces, for $0\le \alpha^\prime<\alpha\le 1$, we thus obtain that the weight function is admissible (see \cite[Example~2.3]{Cuchiero2024}). Therefore, we can apply the weighted Stone--Weierstraß theorem, to obtain a global universal approximation theorem. Hence, let us state the universal approximation theorem on $\mathcal B_\psi(\widehat{\mathcal C}^\pi)$.

\begin{proposition}[Universal approximation theorem on $\mathcal B_\psi(\widehat{\mathcal C}^\pi)$]\label{prop:UAT weighted}
  Let $\psi$ be defined as in \eqref{eq: weight function}. Then, the linear span of the set
  \begin{equation*}
    \Bigl\{\hX^\pi\mapsto \langle e_I,\hbbX^\pi_T\rangle: I\in\{0,\ldots,d\}^N,~N\in\N_0\Bigr\}
  \end{equation*}
  is dense in $\mathcal B_\psi(\widehat{\mathcal C}^\pi)$, i.e., for every map $f\in \mathcal B_\psi(\widehat{\mathcal C}^\pi)$ and every $\epsilon>0$ there exists a linear function $\boldsymbol\ell\colon T((\R^{d+1}))\to\R$ of the form $\hbbX^\pi_T\mapsto \boldsymbol{\ell}(\hbbX^\pi_T):=\sum_{0\le |I|\le N}\ell_I\langle e_I,\hbbX^\pi_T\rangle$, with some $N\in\N_0$ and $\ell_I\in\R$, such that
  \begin{equation*}
    \sup_{\hX^\pi\in\widehat{\mathcal C}^\pi}\frac{|f(\hX^\pi)-\boldsymbol{\ell}(\hbbX^\pi_T)|}{\psi(\hX^\pi)}<\epsilon.
  \end{equation*}
\end{proposition}

\begin{proof}
  The proof follows a similar structure to that of \cite[Theorem~5.4]{Cuchiero2024}, where the authors show the result for weakly geometric $\alpha$-H\"older rough path space. However, since we focus on piecewise linearly interpolated paths, certain arguments require modification. In particular, we will use the equivalence of the norms $|\cdot|_{\R^{d+1}}$ and $\|\cdot\|_{cc}$ on $G^1(\R^{d+1})$ and that for paths of bounded variation, the signature is uniquely determined by the path itself.

  \medskip

  We want to apply the weighted Stone--Weierstraß theorem given in \cite{Cuchiero2024}. Therefore, we have to show that the set $$\mathcal A:=\spn\{\hX^\pi\mapsto \langle e_I,\hbbX^\pi_T\rangle: I\in\{0,\ldots,d\}^N,N\in\N_0\}\subseteq \cB_\psi(\widehat{\mathcal C}^\pi)$$ is a vector subspace and point-separating of $\psi$-moderate growth, which vanishes nowhere. Let us start by showing that $\mathcal A\subseteq \cB_\psi(\widehat{\mathcal C}^\pi)$ is a vector subspace. Therefore, fix some $a\in\mathcal A$ with $\hX^\pi\mapsto a(\hX^\pi):=\langle e_I,\hbbX^\pi_T\rangle\in\R$, for some $I\in\{0,\ldots,d\}^N$ and $N\in\N_0$. For some fixed $R>0$ the pre-image $K_R=\psi^{-1}((0,R])$ is bounded with respect to $\|\cdot\|_\alpha$. As $\cB_\psi(\widehat{\mathcal C}^\pi)$ is the same for each $d_{\alpha^\prime}$, $\alpha^\prime\in [0,\alpha)\cup \{\infty\}$, see \cite{Cuchiero2024}, we show the claim for $d_\infty.$
   
  Note that by Theorem~7.44 in \cite{Friz2010} all homogeneous norms are equivalent on $G^N(\R^{d+1})$. In particular, the Carnot--Carathéodory norm and the norm $\|\mathbf a\|=\max_{i=1,\ldots,N}|\mathbf a^{(i)}|_{(\R^{d+1})^{\otimes i}}^{\frac 1 i}$ are homogeneous and thus equivalent. For $N=1$ this translates to the equivalence of the norms $\|\cdot\|_{cc}$ and $|\cdot|_{\R^{d+1}}$ on $G^1(\R^{d+1})$. Therefore, this yields that the map
  \begin{equation*}
    \iota \colon (\widehat{\mathcal C}^\pi,d_{\infty})\to (C_0^\alpha([0,T];G^1(\R^{d+1})),d_{cc,\infty})\quad\text{via}\quad
    \hX^\pi\mapsto \hbX^\pi,
  \end{equation*}
  is continuous. Moreover, observe that $\hbbX^{\pi,N}$ satisfies the following ODE on $T^N(\R^{d+1})$:
  \begin{equation*}
    \dd \hbbX_t^{\pi,N}=\hbbX_t^{\pi,N}\otimes \dd\hX_t^\pi,\quad
    \hbbX_0^{\pi,N}=\mathbf 1\in G^N(\R^{d+1}).
  \end{equation*}
  This can also be viewed as a rough differential equation (RDE) (see \cite[(C.4)]{Cuchiero2022} for a similar argument), i.e.,
  \begin{equation*}
    \dd \hbbX_t^{\pi,N}=\hbbX_t^{\pi,N}\otimes \dd\hbX_t^\pi,\quad
    \hbbX_0^{\pi,N}=\mathbf 1\in G^N(\R^{d+1}).
  \end{equation*}
  Then, it follows by \cite[Corollary~10.40]{Friz2010} that
  \begin{equation*}
    (\iota(K_R),d_{cc,\infty})\ni \hbX^\pi\mapsto \hbbX^{\pi,N}\in(C_0^{\alpha}([0,T];G^N(\R^{d+1})),d_{cc,\infty})
  \end{equation*}
  is continuous on $\iota(K_R)$. Thus,
  \begin{equation*}
    (K_R,d_\infty)\ni\hX^\pi\mapsto \hbbX^{\pi,N}\in(C_0^\alpha([0,T];G^N(\R^{d+1})),d_{cc,\infty})
  \end{equation*}
  is continuous on $K_R$, as a composition of continuous maps. In addition, we know that the evaluation map
  \begin{equation*}
    (C_0^\alpha([0,T];G^N(\R^{d+1})),d_{cc,\infty})\ni\hbbX^{\pi,N}\mapsto \hbbX_T^{\pi,N}\in(G^N(\R^{d+1}),d_{cc})
  \end{equation*}
  is continuous on $K_R$. Since the composition of continuous functions is again continuous, we obtain that the map
  \begin{equation*}
    (K_R,d_\infty)\ni \hX^\pi\mapsto \hbbX^{\pi,N}_T\in(G^N(\R^{d+1}),d_{cc})
  \end{equation*}
  is continuous on $K_R$. Altogether, we obtain that
  \begin{equation}\label{eq: continuity of a}
    (K_R,d_\infty)\ni\hX^\pi\mapsto a(\hX^\pi)=\langle e_I,\hbbX^\pi_T\rangle\in\R
  \end{equation}
  is continuous on $K_R$ by the linearity of $a$. Hence $a_{|_{K_R}}\in C(K_R)$ for all $R>0$ since $R$ was chosen arbitrarily.

  \medskip

  Using the ball-box estimate \cite[Proposition~7.49, Proposition~7.45]{Friz2010}, we get
  \begin{equation*}
    \|\hbbX_T^{\pi,N}-\hbbX_0^{\pi,N}\|_{T^N(\R^{d+1})}\le C_1\max\Bigl\{d_{cc}(\hbbX_T^{\pi,N},\hbbX_0^{\pi,N}), d_{cc}(\hbbX_T^{\pi,N},\hbbX_0^{\pi,N})^N\Bigr\}
  \end{equation*}
  for some constant $C_1\ge 1$, we have for every $\hX^\pi\in\widehat{\mathcal C}^\pi$,
  \begin{align*}
    |a(\hX^\pi)|&=|\langle e_I,\hbbX_T^\pi\rangle|\le \|\hbbX_T^{\pi,N}\|_{T^N(\R^{d+1})}\le \|\hbbX_T^{\pi,N}-\hbbX_0^{\pi,N}\|_{T^N(\R^{d+1})}+1\\
    &\le C_1( d_{cc}(\hbbX_T^{\pi,N},\hbbX_0^{\pi,N})^N+2).
  \end{align*}
  Using the inequality $d_{cc}(\hbbX_s^{\pi,N},\hbbX_t^{\pi,N})\le C_{N,\alpha}d_{cc}(\hbX^\pi_s,\hbX^\pi_t)$ for some constant $C_{N,\alpha}>0$ (see \cite[Theorem~9.5]{Friz2010}), we obtain
  \begin{align*}
    |a(\hX^\pi)|
    &\le C_1(d_{cc}(\hbbX_T^{\pi,N},\hbbX_0^{\pi,N})^N+2)\\
    &\le C_1\Bigl(T^{\alpha N} \Bigl(\sup_{0\le s<t\le T}\frac{d_{cc}(\hbbX_s^{\pi,N},\hbbX_t^{\pi,N})}{|t-s|^\alpha}\Bigr)^N+2\Bigr)\\
    &\le C_1\Bigl(C_{N,\alpha}^NT^{\alpha N} \Bigl(\sup_{0\le s<t\le T}\frac{d_{cc}(\hbX_s^\pi,\hbX_t^\pi)}{|t-s|^\alpha}\Bigr)^N+2\Bigr).
  \end{align*}
  
  Using the equivalence of the norms $\|\cdot\|_{cc}$ and $|\cdot|_{\R^{d+1}}$ on $G^1(\R^{d+1})$, i.e.~$d_{cc}(\hbX_s^\pi,\hbX_t^\pi)\le C_2|\hX_t^\pi-\hX_s^\pi|_{\R^{d+1}}$ for some constant $C_2>0$, yields
  \begin{align*}
    |a(\hX^\pi)|
    &\le C_1\Bigl(C_{N,\alpha}^NT^{\alpha N} \Bigl(\sup_{0\le s<t\le T}\frac{d_{cc}(\hbX_s^\pi,\hbX_t^\pi)}{|t-s|^\alpha}\Bigr)^N+2\Bigl)\\
    &\le C_1\Bigl(C_{N,\alpha}^NC_2^NT^{\alpha N} \Bigl(\sup_{0\le s<t\le T}\frac{|\hX_t^\pi-\hX_s^\pi|_{\R^{d+1}}}{|t-s|^\alpha}\Bigr)^N+2\Bigl)\\
    &=C_1(C_{N,\alpha}^NC_2^NT^{\alpha N}\|\hX^\pi\|_\alpha^N+2).\\
  \end{align*}
  Thus,
  \begin{equation*}
    \lim_{R\to\infty}\sup_{\hX^\pi\in\widehat{\mathcal C}^\pi\backslash K_R}\frac{|a(\hX^\pi)|}{\psi(\hX^\pi)}
    \le C_1 \sup_{\hX^\pi\in\widehat{\mathcal C}^\pi\backslash K_R}\frac{ \Bigl(C_{N,\alpha}^NC_2^NT^{\alpha N} \|\hX^\pi\|_\alpha^N+2\Bigl)}{\exp(\beta\|\hX^\pi\|_\alpha^\gamma)}=0,
  \end{equation*}
  as the exponential function grows faster than any polynomial. Then, by Lemma~2.7~(ii) in \cite{Cuchiero2024} (or Theorem~2.7 in \cite{Dorsek2010}) it follows that $a\in\cB_\psi(\widehat{\mathcal C}^\pi)$ and thus $\mathcal A\subseteq\cB_\psi(\widehat{\mathcal C}^\pi)$. For the point separation we need to show that
  \begin{align*}
    \tilde{\mathcal A}:=\spn\{\hX^\pi\mapsto \langle (e_I\shuffle e_0^{\otimes k})\otimes e_0,\hbbX^\pi_T\rangle: I\in\{0,\ldots,d\}^N,k\in\N_0, N\in\{0,1\}\}\subseteq\mathcal A
  \end{align*}
  is point-separating such that $\hX^\pi\mapsto\exp(|\tilde a(\hX^\pi)|)\in \cB_\psi(\widehat{\mathcal C}^\pi)$ for all $\tilde a\in\cB_\psi(\widehat{\mathcal C}^\pi)$. In contrast to \cite{Cuchiero2024} we consider paths of bounded variation, where higher order lifts are uniquely determined by the first order. Therefore, we have to show that for $\hX^\pi,\hY^\pi\in\widehat{\mathcal C}^\pi$ distinct that there exists some $k\in\N_0$, $N\in\{0,1\}$ and $I\in\{0,\ldots,d\}^N$, such that 
  \begin{equation*}
    \langle (e_I\shuffle e_0^{\otimes k})\otimes e_0,\hbbX^\pi_T\rangle\neq\langle (e_I\shuffle e_0^{\otimes k})\otimes e_0,\hbbY^\pi_T\rangle.
  \end{equation*}
  
  As in \cite{Cuchiero2024} we can proceed by a proof of contradiction. Thus assume that for every $k\in\N_0$, $N\in\{0,1\}$ and $I\in\{0,\ldots,d\}^N$ it holds that
  \begin{equation*}
    \langle (e_I\shuffle e_0^{\otimes k})\otimes e_0,\hbbX^\pi_T\rangle=\langle (e_I\shuffle e_0^{\otimes k})\otimes e_0,\hbbY^\pi_T\rangle.
  \end{equation*}
  By the shuffle property, we observe that
  \begin{align*}
    \langle (e_I\shuffle e_0^{\otimes k})\otimes e_0,\hbbX^\pi_T\rangle&=\int_0^T\langle e_I\shuffle e_0^{\otimes k},\hbbX^\pi_t\rangle\dd t=\int_0^T\langle e_I,\hbbX^\pi_t\rangle\langle e_0^{\otimes k},\hbbX^\pi_t\rangle\dd t\\ \nonumber
    &=\int_0^T\langle e_I,\hbbX^\pi_t\rangle\frac{t^k}{k!}\dd t,
  \end{align*}
  for all $\hX^\pi\in\widehat{\mathcal C}^\pi$. Hence, we can conclude that for every $k\in\N_0$, $N\in\{0,1\}$ and $I\in\{0,\ldots,d\}^N$ we obtain
  \begin{equation*}
    \int_0^T\langle e_I,\hbbX^\pi_t-\hbbY^\pi_t\rangle\frac{t^k}{k!}\dd t=0.
  \end{equation*}
  By \cite[Corollary~4.24]{Brezis2011}, we then deduce that $\langle e_I,\hbbX_t^\pi\rangle=\langle e_I,\hbbY_t^\pi\rangle$ for all $I\in\{0,\ldots,d\}^N$, $N\in\{0,1\}$. This contradicts the assumption that $\hX^\pi,\hY^\pi$ are distinct. Thus, we conclude that $\tilde{\mathcal A}$ is point-separating. Therefore, as we consider continuous paths of bounded variation, it is enough to consider $N\in\{0,1\}$ and $k\in\N_0$.
  
  Lastly, we have to show that $\exp(\lambda|\tilde a(\cdot)|)\in\mathcal B_\psi(\widehat{\mathcal C}^\pi)$. This is precisely the step where the choice of the exponential weight function, as given in \eqref{eq: weight function}, becomes crucial. Therefore, let us fix some $(\hX^\pi\mapsto \tilde a(\hX^\pi)=\ell(\hbbX^\pi_T))\in\tilde{\mathcal A}$ with
  \begin{equation*}
    \ell(\hbbX^\pi_T)=\sum_{|I|\le N}\sum_{k=0}^K a_{I,k}\langle(e_I\shuffle e_0^{\otimes k})\otimes e_0,\hbbX^\pi_T\rangle,
  \end{equation*}
  for some $K\in\N_0$, $N\in\{0,1\}$ and $a_{I,k}\in\R$. Then, by similar arguments as for \eqref{eq: continuity of a}, we have $\exp(|\tilde a(\cdot)|)_{|_{K_R}}\in C(K_R)$ for all $R>0$.

  Moreover, using the ball-box estimate, as before, it follows that
  \begin{align*}
    |\tilde a(\hX^\pi)|=|\ell(\hbbX_T^\pi)|\le C_1\|\ell\|_{T^{N+K+1}(\R^{d+1})^\ast}\Bigl(T^{\alpha(K+1)N}C_{N,\alpha}^N\Bigl(\sup_{0\le s<t\le T}\frac{d_{cc}(\hbX_s^\pi,\hbX_t^\pi)}{|t-s|^\alpha}\Bigr)^N+1\Bigr),
  \end{align*}
  for some constants $C_1,C_{N,\alpha}>0$. Using the equivalence of the norms $\|\hbX^\pi\|_{cc}$ and $|\hX^\pi|_{\R^{d+1}}$ with some constant $C_2>0$, yields
  \begin{align*}
    &C_1\|\ell\|_{T^{N+K+1}(\R^{d+1})^\ast}\Bigl(T^{\alpha(K+1)N}C_{N,\alpha}^N\Bigl(\sup_{0\le s<t\le T} \frac{d_{cc}(\hbX_s^\pi,\hbX_t^\pi)}{|t-s|^\alpha}\Bigr)^N+1\Bigr)\\
    &\quad\le C_1\|\ell\|_{T^{N+K+1}(\R^{d+1})^\ast}\Bigl(T^{\alpha(K+1)N}C_{N,\alpha}^NC^N_2\Bigl(\sup_{0\le s<t\le T}\frac{|\hX_t^\pi-\hX_s^\pi|_{\R^{d+1}}}{|t-s|^\alpha}\Bigr)^N+1\Bigr)\\
    &\quad=C_1\|\ell\|_{T^{N+K+1}(\R^{d+1})^\ast}\Bigl(T^{\alpha(K+1)N}C_{N,\alpha}^NC^N_2\|\hX^\pi\|^N_{\alpha}+1\Bigr).\\
  \end{align*}
  Then, for $C:=\max(C_1C_{N,\alpha}^NC^N_2\|\ell\|_{T^{N+K+1}(\R^{d+1})^\ast}T^{\alpha(K+1)N},C_1\|\ell\|_{T^{N+K+1}(\R^{d+1})^\ast})$, we obtain
  \begin{equation*}
    \lim_{R\to\infty}\sup_{\hX^\pi\in\widehat{\mathcal C}^\pi\backslash K_R}\frac{\exp(\lambda|\tilde a(\hX^\pi)|)}{\psi(\hX^\pi)}\le \lim_{R\to\infty}\sup_{\hX^\pi\in\widehat{\mathcal C}^\pi\backslash K_R}\exp(\lambda C(\|\hX^\pi\|^N_{\alpha}+1)-\beta\|\hX^\pi\|_{\alpha}^\gamma)=0,
  \end{equation*}
  for $\lambda<\frac{\beta}{C}$ small enough and since the exponent tends to $-\infty$ as $R\to\infty$ and using that $\gamma\ge 1\ge N$, the last equality follows and thus we have by Lemma~2.7~(ii) in \cite{Cuchiero2024} that $\exp(\lambda|\tilde a(\cdot)|)\in\cB_\psi(\widehat{\mathcal C}^\pi)$ for any $\tilde a\in\tilde{\mathcal A}.$

  Lastly, we show that $\tilde\cA$ vanishes nowhere. Using the map $(\hX^\pi\mapsto\tilde a(\hX^\pi):=\langle (e_\emptyset\shuffle e_0^{\otimes 0})\otimes e_0,\hbbX^\pi_T\rangle)\in\tilde\cA$, then $\tilde a(\hX^\pi)=\int_0^T\dd t=T\neq 0$, for all $\hX^\pi\in\widehat{\mathcal C}^\pi$. Then, by the weighted real-valued Stone--Weierstraß theorem, it follows that $\cA$ is dense in $\cB_\psi(\widehat{\mathcal C}^\pi)$. All in all, we see that the proof is very similar to the proof of Theorem~5.4 in \cite{Cuchiero2024}.
\end{proof}

We aim to establish a similar result for the space of stopped paths, which plays a fundamental role in functional It\^o calculus, see \cite{Cont2013, Dupire2019}, and also in rough path settings, see e.g.~\cite{Bayer2025,Cuchiero2024b}. To this end, let us first introduce this space and recall the notion of a non-anticipative path functional.

\medskip

The space of stopped paths in $\widehat{\mathcal C}^\pi$ is defined by
\begin{equation*}
  \Lambda^\pi_T:=\bigcup_{t\in[0,T]}\widehat{\mathcal C}^\pi_t:=\bigcup_{t\in[0,T]}\{\hX^\pi_{[0,t]} \in\widehat{\mathcal C}^\pi: \hX^\pi_s=(s,X^\pi_s)~\text{for all}~s\in[0,t]\}.
\end{equation*}

We equip this space with the metric (cf.~\cite[Section~2.2]{Cont2013})
\begin{equation*}
  d_{\Lambda,\alpha^\prime}(\hX^\pi_{[0,t]},\hY^\pi_{[0,s]}):=|t-s|+d_{\alpha^\prime}(\hX_{[0,T]}^{\pi,t},\hY_{[0,T]}^{\pi,s}), 
\end{equation*}
for some $0\le \alpha^\prime<\alpha$, where we denote by $\hX^{\pi,t}$ the process $(\hX^{\pi,t})_u=(u,X^{\pi,t}_u)$, i.e., the process where $X^\pi$ is stopped at time $t$ but the time-augmentation is not stopped.
  
Moreover, we call a map $f\colon \Lambda^\pi_T\to\R$ a non-anticipative functional if it is measurable. It is called continuous if $f\colon \Lambda^\pi_T\to\R$ is continuous with respect to the metric $d_{\Lambda,\alpha^{\prime}}$. We further define the quotient map
\begin{equation}\label{eq: quotient map}
  \phi\colon [0,T]\times\widehat{\mathcal C}^\pi\to\Lambda^\pi_T,\qquad \phi(t,\hX^\pi)=\hX^\pi_{[0,t]},
\end{equation}
where $\hX^\pi_{[0,t]}$ denotes the restriction of $\hX^\pi$, which is defined on $[0,T]$, to the sub-interval $[0,t], t\le T$. Moreover, the map $\phi$ is continuous. The proof follows analogously to \cite[Lemma~A.3]{Ceylan2025}, with the rough path metric replaced by the H\"older metric. Furthermore, by similar arguments as in \cite[Lemma~A.4]{Ceylan2025} it follows that $\Lambda_T^\pi$ is a Polish space. Therefore, we omit the proof.

To obtain a global universal approximation result on $\Lambda_T^\pi$, we need to show that $(\Lambda_T^\pi,\psi)$ is a weighted space. To that end, we consider the following weight function
\begin{equation}\label{eq: weight function 2}
  \psi(\hX^\pi_{[0,t]}):=\exp(\beta\|\hX^{\pi,t}_{[0,T]}\|_\alpha^\gamma)
\end{equation}
with $\gamma\ge 1$ and $\beta>0$.

\begin{lemma}\label{lem: weighted space}
  Let $0\le \alpha^{\prime}<\alpha\le 1$ and let $\psi$ be defined as in \eqref{eq: weight function 2}. Then $K_R:=\psi^{-1}((0,R])=\{\hX^\pi_{[0,t]}\in\Lambda^\pi_T: \psi(\hX^\pi_{[0,t]})\le R\}$ is compact w.r.t.~the topology induced by $d_{\Lambda,\alpha^\prime}$ and $(\Lambda_T^\pi,\psi)$ is a weighted space.
\end{lemma}

\begin{proof}
Since $\psi\geq 1$, the assertion is immediate for $0<R<1$.
Let therefore $R\geq1$ and let $(\lambda_n)_{n\in\N}$ be an arbitrary sequence in $K_R$. For every $n\in\N$ there exists a $t_n\in[0,T]$ and $\X^\pi_n\in\widehat{\mathcal C}^\pi$ such that $\lambda_n=\hX^\pi_{n,[0,t_n]}$. Denote by $\hX^{\pi,t_n}$ the canonical extension of $\lambda_n$ to $[0,T]$, defined by $(\hX^{\pi,t_n})_u=(u,\hX^{\pi,t_n}_{n,u})$, $u\in[0,T]$. Since $\lambda_n\in K_R$ we have $\psi(\lambda_n)\le R$ and thus $\|\lambda_n\|_\alpha\le (\frac{\log(R)}{\beta})^{\frac{1}{\gamma}}$, for every $n\in\N$. After passing to a subsequence, compactness of $[0,T]$ and the compact embedding of $\alpha$-H\"older spaces into $\alpha^\prime$-H\"older spaces, $\alpha^\prime<\alpha$, yield $t_n\to t$, $\hX^{\pi,t_n}_n\to Y$ in $C^{\alpha^\prime}([0,T];\R^{d+1})$. Since $\pi$ is fixed and finite, the limit $Y$ is again the canonical extension of a stopped piecewise linear path, thus $Y=\hX^{\pi,t}$ for some $\hX^{\pi}_{[0,t]}\in \Lambda_T^\pi$. Moreover, by the lower semicontinuity of the $\alpha$-H\"older norm we have $\|\hX^{\pi,t}\|_{\alpha}\le \liminf_{n\to\infty}\|\hX_n^{\pi,t_n}\|_\alpha\le (\frac{\log(R)}{\beta})^{\frac{1}{\gamma}}$. Therefore, $\hX^{\pi,t}\in K_R$. Finally, $d_{\Lambda,\alpha^\prime}(\lambda_n,\hX^{\pi}_{[0,t]})=|t-t_n|+d_{\alpha^\prime}(\hX^{\pi,t_n}_n,\hX^{\pi,t})\to 0$. Thus $K_R$ is sequentially compact and since $d_{\Lambda,\alpha^\prime}$ is a metric, compact. Consequently, $(\Lambda_T^\pi,\psi)$ is a weighted space.
\end{proof}

\begin{proposition}[Universal approximation theorem on $\cB_\psi(\Lambda_T^\pi$)]\label{prop:UAT non-anticipative}
  Let $\psi$ be defined as in \eqref{eq: weight function 2}. Then, the linear span of the set
  \begin{equation*}
    \Bigl\{ \hX^\pi_{[0,t]}\to\langle e_I,\hbbX^\pi_t\rangle: I\in \{0,\ldots,d\}^N, N\in \N_0\Bigr\}
  \end{equation*}
  is dense in $\cB_\psi(\Lambda_T^\pi)$, i.e., for every map $f\in \cB_\psi(\Lambda_T^\pi)$ and every $\epsilon>0$, there exists a linear function $\boldsymbol\ell\colon T((\R^{d+1}))\to\R$ of the form $\hbbX^\pi_t\mapsto\boldsymbol{\ell}(\hbbX^\pi_t):=\sum_{|I|\le N}\ell_I\langle e_I, \hbbX^\pi_t\rangle$, with some $N\in \N_0$ and $\ell_I\in\R$, such that
  \begin{equation*}
    \sup_{\hX^\pi_{[0,t]}\in \Lambda_{T}^\pi}\frac{|f(\hX^\pi_{[0,t]})-\boldsymbol{\ell}(\hbbX^\pi_t)|}{\psi(\hX^\pi_{[0,t]})}<\epsilon.
  \end{equation*}
\end{proposition}

\begin{proof}
  First note that Lemma~\ref{lem: weighted space} implies that $(\Lambda_T^\pi,\psi)$ is a weighted space. Hence, we may apply the weighted Stone--Weierstraß theorem introduced in \cite[Theorem~3.9]{Cuchiero2024}. It therefore suffices to verify that $\mathcal A:=\spn\{\hX^\pi_{[0,t]}\mapsto \langle e_I,\hbbX_t^\pi\rangle: I\in\{0,\ldots,d\}^N, N\in\N_0\}\subseteq\cB_\psi(\Lambda_T^\pi)$ is a vector subspace and point-separating of $\psi$-moderate growth, which vanishes nowhere. To that end, we need to prove that
  \begin{align*}
    \widetilde{\cA}
    &:= \spn\Bigl(
      \Bigl\{
        \hX^\pi_{[0,t]}
        \mapsto
        \langle e_\emptyset,\hbbX^\pi_t\rangle
      \Bigr\}\nonumber\\
      &\qquad \cup
      \Bigl\{
        \hX^\pi_{[0,t]} \mapsto
        \langle ( e_I \shuffle e_0^{\otimes k} ) \otimes e_0,
        \hbbX^\pi_t \rangle :
        \begin{matrix}
          k \in \N_0, \,
          N \in \{0, 1 \},\\
          I \in \{0,\ldots,d\}^N
        \end{matrix}
      \Bigr\}
    \Bigr)\subseteq \cA\nonumber,
  \end{align*}
  is a point-separating and nowhere vanishing vector subspace of $\psi$-moderate growth. The claim follows by combining the arguments of Proposition~\ref{prop:UAT weighted} with those of \cite[Proposition~3.11]{Ceylan2025}, and we therefore omit the details.
\end{proof}

\subsection{Approximation of \texorpdfstring{$L^p$}{-}-functionals via signatures}

We present global universal approximation theorems for $L^p$-functionals. Building on the universal approximation results for weighted function spaces from the previous subsection (Proposition~\ref{prop:UAT weighted} and Proposition~\ref{prop:UAT non-anticipative}), we now additionally impose an integrability condition on the weight function $\psi$, thereby extending the result to $L^p$-functionals.

\begin{theorem}[$L^p$-convergence]\label{thm:Lpmain}
  Let $p\ge1$.
  \begin{enumerate}
    \item[(i)]  Let $(\widehat{\mathcal C}^\pi,\cB(\widehat{\mathcal C}^\pi))$ be a Borel space with finite measure $\nu$. Let $\psi$ be defined as in \eqref{eq: weight function} and $\int_{\widehat{\mathcal C}^\pi}\psi^p\dd\nu<\infty$. Moreover, consider the set
    \begin{equation*}
      \mathcal L:=\Bigl\{f_\ell|~f_\ell\colon \hX^\pi\mapsto \boldsymbol\ell(\hbbX_T^\pi)=\sum_{|I|\le N}\ell_I\langle e_I,\hbbX_T^\pi\rangle,~\ell_I\in\R,~N\in\N_0,~\hX^\pi\in\widehat{\mathcal C}^\pi\Bigr\}.
    \end{equation*}
    Then, for every $f\in L^p(\widehat{\mathcal C}^\pi;\nu)$ and for every $\epsilon>0$, there exists a functional $f_\ell\in\mathcal L$ such that
    \begin{equation*}
      \|f-f_\ell\|_{L^p(\widehat{\mathcal C}^\pi)}<\epsilon.
    \end{equation*}
    \item[(ii)] Let $(\Lambda_T^\pi,\cB(\Lambda_T^\pi))$ be a Borel space with finite measure $\nu$. Let $\psi$ be defined as in \eqref{eq: weight function 2} and $\int_{\Lambda_T^\pi}\psi^p\dd\nu<\infty$. Moreover, consider the set
    \begin{equation*}
      \mathcal L_\Lambda:=\Bigl\{f_\ell|~f_\ell\colon \hX_{[0,t]}^\pi\mapsto \boldsymbol\ell(\hbbX_t^\pi)=\sum_{|I|\le N}\ell_I\langle e_I,\hbbX_t^\pi\rangle,~\ell_I\in\R,~N\in\N_0,~\hX_{[0,t]}^\pi\in\Lambda_T^\pi\Bigr\}.
    \end{equation*}
    Then, for every $f\in L^p(\Lambda_T^\pi;\nu)$ and for every  $\epsilon>0$, there exists a functional $f_\ell\in\mathcal L_\Lambda$ such that
    \begin{equation*}
      \|f-f_\ell\|_{L^p(\Lambda_T^\pi)}<\epsilon.
    \end{equation*}
  \end{enumerate}
\end{theorem}

\begin{proof}
  (i) Let $f\in L^p(\widehat{\mathcal C}^\pi;\nu)$. Fix $\epsilon>0$.
  
  \textit{Step~1}: For any $K>0$, we can define the function $f_K(x):=1_{\{|f(x)|\le K\}}(x)f(x)$ for which we have $\|f-f_K\|_{L^p(\widehat{\mathcal C}^\pi)}\to 0$ as $K\to \infty$ by dominated convergence. Therefore, there is a $K^{\epsilon}>0$ such that $\|f-f_{K^{\epsilon}}\|_{L^p(\widehat{\mathcal C}^\pi)}\le\frac{\epsilon}{3}$.
  
  \textit{Step~2}: By Lusin's theorem \cite[Theorem~2.5.17]{Denkowski2003}, there is a closed set $C^\epsilon\subset \widehat{\mathcal C}^\pi$, such that $f_{K^\epsilon}$ restricted to $C^\epsilon$ is continuous and $\nu(\widehat{\mathcal C}^\pi\setminus C^\epsilon)\le \frac{\epsilon^p}{(6 K^\epsilon)^p}$. By Tietze's extension theorem \cite[Theorem~3.6.3]{friedman1982}, there is a continuous extension $f^\epsilon\in C_b(\widehat{\mathcal C}^\pi,[-K^\epsilon,K^\epsilon])$ of $f_{K^\epsilon},$ such that
  \begin{equation*}
    \|f_{K^\epsilon}-f^\epsilon\|_{L^p(\widehat{\mathcal C}^\pi)}^p=\int_{\widehat{\mathcal C}^\pi\setminus C^\epsilon}|f_{K^\epsilon}-f^\epsilon|^p\dd\nu\le(2K^\epsilon)^p\nu(\widehat{\mathcal C}^\pi\setminus C^\epsilon)\le \Bigl(\frac{\epsilon}{3}\Bigr)^p.
  \end{equation*}

  \textit{Step~3}: Moreover, since by the definition of the weighted function space $ \cB_\psi(\widehat{\mathcal C}^\pi)$ it holds that $C_b(\widehat{\mathcal C}^\pi)\subseteq \cB_\psi(\widehat{\mathcal C}^\pi)$, by Proposition~\ref{prop:UAT weighted}  we can approximate $f^\epsilon$ by a linear combination of the signature. More precisely, setting $M=\int_{\widehat{\mathcal C}^\pi}\psi^p\dd\nu<\infty$ we have
  \begin{equation*}
    \|{f}^\epsilon-f_\ell\|^p_{\cB_\psi(\widehat{\mathcal C}^\pi)}=\Bigl(\sup_{\hX^\pi\in\widehat{\mathcal C}^\pi}\frac{|{f}^\epsilon(\hX^\pi)-\boldsymbol{\ell}(\hbbX_T^\pi)|}{\psi(\hX^\pi)}
    \Bigr)^p< \frac{\epsilon^p}{3^p M}.
  \end{equation*}
  Hence, we get
  \begin{equation*}
     \|{f}^\epsilon-f_\ell\|^p_{L^p(\widehat{\mathcal C}^\pi)}\le\int_{\widehat{\mathcal C}^\pi}\psi^p\dd\nu~\|{f}^\epsilon-f_\ell\|^p_{\cB_\psi(\widehat{\mathcal C}^\pi)}< \Bigl(\frac{\epsilon}{3}
        \Bigr)^p.
  \end{equation*}
  Altogether, we obtain
  \begin{equation*}
    \|f-f_\ell\|_{L^p(\widehat{\mathcal C}^\pi)}\le\|f-f_{K^\epsilon}\|_{L^p(\widehat{\mathcal C}^\pi)}+\|f_{K^\epsilon}-f^\epsilon\|_{L^p(\widehat{\mathcal C}^\pi)}+
     \|{f}^\epsilon-f_\ell\|_{L^p(\widehat{\mathcal C}^\pi)}< \epsilon,
  \end{equation*}
  which concludes the proof.
  
  (ii) The proof follows along the same lines as in (i), where in \textit{Step~3}, we use Proposition~\ref{prop:UAT non-anticipative}.
\end{proof}

\begin{remark}\label{rem: other Lp results}
  Note that the assumption $\int_{\widehat{\mathcal C}^\pi} \psi^p \dd\nu < \infty$, with $\psi(\hX^\pi) = \exp\Big(\beta\|\hX^\pi\|_{\alpha}^\gamma\Big)$, represents an exponential moment condition. Such a condition is crucial for establishing \(L^p\)-convergence with the method developed here.

  While our analysis is carried out in a discrete-time setting, several recent works establish $L^p$-universal approximation theorems on spaces of weakly geometric rough paths. In~\cite{Ceylan2025}, a similar $L^p$-universal approximation theorem is derived within the framework of the geometric $\alpha$-H\"older rough path space, assuming the $p$-integrability of the weight function $\psi(\hbX)=\exp(\beta\|\hbX\|_{cc,\alpha}^\gamma)$ on rough path space. Despite the similarity of the two settings, our discrete-time results do not follow directly from \cite{Ceylan2025}. Indeed, while the $L^p$-approximation result can formally be transferred via pushforward arguments together with the canonical embedding of the space of piecewise linearly interpolated paths into the space of geometric $\alpha$-H\"older rough paths $\widehat C_{d,T}^\alpha$, this argument does not extend to the weighted setting. The reason is that the weighted universal approximation theorem requires the corresponding functional to belong to the weighted space $\cB_\psi(\widehat{C}_{d,T}^\alpha)$, which entails uniform control in the weighted supremum norm induced by the weight function $\psi$. After composition with the canonical embedding map, however, such control is obtained only on the image of the embedding. In general, this does not imply that the induced functional remains controlled by the weight function on the entire space $\widehat{C}_{d,T}^\alpha$. Consequently, the weighted structure is not preserved when passing from continuous to discretely generated path spaces, and a separate discrete-time analysis is required.
  
   We also note that $L^p$-approximation results on the geometric rough path space have been obtained under a different type of assumptions, namely under an infinite radius of convergence condition, see \cite{Chevyrev2026}. Similarly, there are $L^p$-type results for so-called robust signatures, which were introduced in \cite{Chevyrev2022}. However, compared to our approach in \cite{Bayer2025} they use a Stone--Weierstraß argument for robust signatures, and in \cite{Schell2023} they use a monotone class argument.
\end{remark}

\section{Universal approximation via discrete-time signatures of Gaussian processes}\label{sec: applications}

In this section, we show that the piecewise linear interpolation of a Gaussian process satisfies the required integrability condition, which yields $L^p$-approximation results for measurable and continuous Gaussian path functionals by linear functionals of the signature of its piecewise linear interpolation. In particular, this holds for (fractional) Brownian motion. Moreover, we discuss convergence rates for piecewise linear approximations of Gaussian signatures.

\medskip

Throughout this section, let $X=(X_t)_{t\in[0,T]}$ be a $d$-dimensional continuous, centred Gaussian process with independent components defined on a probability space $(\Omega,\mathcal F,\P)$, with a filtration $(\cF_t)_{t\in[0,T]}$, satisfying the usual conditions, i.e., completeness and right-continuity.
We denote by
\begin{align*}
  R_X\colon[0,T]^2&\to\R^{d\times d}\\
  (s,t)&\mapsto\E[X_s\otimes X_t],
\end{align*}
the covariance function of $X$. For the $i$-th component of $X$, we write $R_{X^i}:=\E[X_s^iX_t^i]$, $0\le s,t\le T$. In what follows, we assume that there exists a $\varrho\in[1,2)$ and $M<\infty$ such that, for every $i\in\{1,\ldots,d\}$ and all $0\le s\le t\le T$,
\begin{equation*}
  \|R_{X^i}\|_{\varrho;[s,t]^2}\le M|t-s|^{1/\varrho},
\end{equation*}
where $\|\cdot\|_{\varrho;I\times I^\prime}$ denotes the two-dimensional $\varrho$-variation on a rectangle $I\times I^\prime$ as defined in \cite[(10.5)]{Friz2020}.

Moreover, we consider a nested sequence of partitions $\pi_n = \{t_i^n\}_{i=0}^{N_n}$ of $[0,T]$ with vanishing mesh size, i.e. $\pi_n \subset \pi_{n+1}$ and $|\pi_n| \to 0$ as $n \to \infty$, and denote by $X^{\pi_n}$ the piecewise linear interpolation of $X$ along $\pi_n$.

\subsection{Approximation of path functionals}\label{subsec: Gaussian processes}

We apply the universal approximation result of Theorem~\ref{thm:Lpmain} to stochastic path functionals of Gaussian processes.
More precisely, we show that measurable and continuous (possibly non-anticipative) functionals of a Gaussian path can be approximated in 
$L^p$ by linear functionals of the signature of its piecewise linear interpolation.

\begin{proposition}
  Let $p\ge 1$, $\varrho\in[1,2)$ and let $\alpha\in (\frac{1}{4},\frac{1}{2\varrho})$. Let $X$ be a $d$-dimensional continuous, centred Gaussian process with independent components and covariance $R_X$ such that there exists a $\varrho\in[1,2)$ and $M<\infty$ such that, for every $i\in\{1,\ldots,d\}$ and all $0\le s\le t\le T$,
  \begin{equation*}
    \|R_{X^i}\|_{\varrho;[s,t]^2}\le M|t-s|^{1/\varrho}.
  \end{equation*}
  Fix $n\in\N$ and let $X^{\pi_n}$ denote the piecewise linear interpolation of $X$ along $\pi_n$. Set $\hX_t:=(t,X_t)$ and $\hX^{\pi_n}_t:=(t,X^{\pi_n}_t)$, $t\in[0,T]$.
  \begin{enumerate}
    \item[(i)] Let  $f(\hX^{\pi_n})\in L^p(\Omega)$ with $f\colon \widehat{\mathcal C}^{\pi_n}\to\R$. Then for every $\epsilon>0$ there exists a linear functional $\boldsymbol\ell\colon T((\R^{d+1}))\to\R$ of the form $\hbbX^{\pi_n}_T\mapsto\sum_{|I|\le N}\ell_I\langle e_I,\hbbX^{\pi_n}_T\rangle$, for some $N\in\N_0$ and $\ell_I\in\R$, such that
    \begin{equation*}
      \E[|f(\hX^{\pi_n})-\boldsymbol\ell(\hbbX_T^{\pi_n})|^p]<\epsilon^p.
    \end{equation*}
    \item[(ii)] Let $f\in L^p(\Lambda_T^{\pi_n})$. Then for every $\epsilon>0$ there exists a linear functional $\boldsymbol\ell\colon T((\R^{d+1}))\to\R$ of the form $\hbbX_t^{\pi_n}\mapsto\sum_{|I|\le N}\ell_I\langle e_I,\hbbX^{\pi_n}_t\rangle$, for some $N\in\N_0$ and $\ell_I\in\R$, such that
    \begin{equation*}
      \E\Bigl[\int_0^T \big|f(\hX^{\pi_n}_{[0,t]})-\boldsymbol\ell(\hbbX^{\pi_n}_t)\big|^p\dd t\Bigr] < \epsilon^p.
    \end{equation*}
  \end{enumerate}
\end{proposition}

\begin{proof}
  (i) We apply Theorem~\ref{thm:Lpmain}. To do so, it remains to verify the exponential moment condition for the law $\nu_n:=\mu_{\hX^{\pi_n}}$, i.e. $\E[\exp(p\beta \|\hX^{\pi_n}\|_\alpha^\gamma)]<\infty$ for $\gamma\ge 1$ and $\beta>0$. 

  To that end, first observe that, by \cite[Proposition~15.11]{Friz2010} the covariance of $X^{\pi_n}$ satisfies for every $i\in\{1,\ldots,d\}$ and $0\le s\le t\le T$
  \begin{equation*}
    \|R_{(X^i)^{\pi_n}}\|_{\varrho;[s,t]^2}\le C\|R_{X^i}\|_{\varrho;[s,t]^2}\le CM|t-s|^{1/\varrho},
  \end{equation*}
  with some constant $C=C(\varrho)>0$. Then by \cite[Corollary~15.29]{Friz2010}, there exists $\eta=\eta(C,M,T,\alpha,\varrho)>0$ such that
  \begin{equation*}
  \E[\exp(\eta \|\hbbX^{\pi_n,3}\|_{cc,\alpha}^2)]<\infty.
  \end{equation*}
  Moreover, as mentioned in the proof of Proposition~\ref{prop:UAT weighted}, on $G^1(\R^{d+1})$ the norm $|\cdot|_{\R^{d+1}}$ is equivalent to the norm $\|\cdot\|_{cc}$ for some constant $C_1>0$. Using this, we observe that
  \begin{align}\label{eq: estimate linear interpolation}
    \|\hX^{\pi_n}\|_\alpha&=\sup_{0\le s<t\le T}\frac{|\hX^{\pi_n}_{s,t}|_{\R^{d+1}}}{|t-s|^\alpha}\le C_1\sup_{0\le s<t\le T}\frac{\|\hbX^{\pi_n}_{s,t}\|_{cc}}{|t-s|^\alpha}\le C_1\sup_{0\le s<t\le T}\frac{\|\hbbX^{\pi_n,3}_{s,t}\|_{cc}}{|t-s|^\alpha}.
  \end{align}
  
  This yields, that the exponential moment condition is fulfilled if we consider the linear interpolation of a Gaussian process, that is,
  \begin{equation}\label{eq: integrability verification}
    \int_{\widehat{\mathcal C}^{\pi_n}}\psi^p\dd\nu_n=\E\Bigl[\exp\Bigl(\beta p\|\hX^{\pi_n}\|_{\alpha}^2\Bigr)\Bigr]\le\E\Bigl[\exp\Bigl(C_1^2\beta p\| \hbbX^{\pi_n,3}\|_{cc,\alpha}^2\Bigr)\Bigr]<\infty,
  \end{equation}
  by taking $\beta\in (0,\frac{\eta}{C_1^2 p}],$ where $C_1>0$ as in \eqref{eq: estimate linear interpolation}. Further, since $f(\hX^{\pi_n})\in L^p(\Omega)$,
  \begin{equation*}
    \int_{\widehat{\mathcal C}^{\pi_n}}|f|^p\dd\nu_n=\E[|f(\hX^{\pi_n})|^p]<\infty,
  \end{equation*}
  that is, $f\in L^p(\widehat{\mathcal C}^{\pi_n})$. Hence, Theorem~\ref{thm:Lpmain} yields that there exists a functional $f_\ell\in\mathcal L$ such that $\|f-f_\ell\|_{L^p(\widehat{\mathcal C}^{\pi_n})}<\epsilon$,
  and, therefore,
  \begin{equation*}
    \E[|f(\hX^{\pi_n})-\boldsymbol\ell(\hbbX^{\pi_n}_T)|^p]=\|f-f_\ell\|^p_{L^p(\widehat{\mathcal C}^{\pi_n})}<\epsilon^p.
  \end{equation*}

  (ii) Let $\nu_n:=(dt\otimes\mu_{\hX^{\pi_n}})\circ\phi^{-1}$ be the push-forward measure on $(\Lambda_T^{\pi_n},\mathcal B(\Lambda_T^{\pi_n}))$, where $\phi(t,\hX^{\pi_n})=\hX^{\pi_n}_{[0,t]}$ is the quotient map as given in \eqref{eq: quotient map}.

  Note that, by the continuity of $\phi$, $\nu_n$ is then a well-defined Borel measure. Moreover, by a change of measure result, we observe that
  \begin{align}
    \int_{\Lambda_T^{\pi_n}}\psi^p\dd\nu_n&=\int_{\widehat{\mathcal C}^{\pi_n}}\int_0^T (\psi\circ\phi)(t,\hX^{\pi_n})\dd t\dd\mu_{\hX^{\pi_n}}\nonumber\\
    &=\E\Bigl[\int_0^T\psi(\hX^{\pi_n}_{[0,t]})^p\dd t\Bigr]\nonumber\\
    &=\E\Bigl[\int_0^T\exp(p\beta\|\hX^{\pi_n,t}_{[0,T]}\|_\alpha^\gamma)\dd t\Bigr]\nonumber\\
    &\le T \E\Bigl[\sup_{t\in[0,T]}\exp(p\beta\|\hX^{\pi_n,t}_{[0,T]}\|_\alpha^\gamma)\Bigr]\label{eq: integrability verification 2}\\
    &=T \E\Bigl[\exp(p\beta\|\hX^{\pi_n}_{[0,T]}\|_\alpha^\gamma)\Bigr]\nonumber\\
    &\le T\E\Bigl[\exp\Bigl(C_1^\gamma\beta p\| \hbbX^{\pi_n,3}\|_{cc,\alpha}^\gamma\Bigr)\Bigr]<\infty\nonumber,
  \end{align}
  for $\beta\in (0,\frac{\eta}{C_1^\gamma p}]$ and $\gamma=2$, where $C_1>0$ as in \eqref{eq: estimate linear interpolation} and where we used
  \begin{equation*}
    \sup_{t\in[0,T]}\|\hX^{\pi_n,t}_{[0,T]}\|_\alpha=\|\hX^{\pi_n}_{[0,T]}\|_\alpha.
  \end{equation*}

  Thus, the integrability condition given in Theorem~\ref{thm:Lpmain} is satisfied, and therefore there exists a functional $f_\ell\in\mathcal L_\Lambda$ such that
  \begin{equation*}
    \|f-f_\ell\|_{L^p(\Lambda_T^{\pi_n})}^p< \epsilon^p.
  \end{equation*}
  Thus, there exists a linear functional $\boldsymbol\ell$ on the signature of $\hX^{\pi_n}$ such that
  \begin{equation*}
    \E\Bigl[\int_0^T|f(\hX^{\pi_n}_{[0,t]})-\boldsymbol\ell(\hbbX_t^{\pi_n})|^p\dd t\Bigr]=\int_{\widehat{\mathcal C}^{\pi_n}}\int_0^T|f-f_\ell|^p\dd t\dd\mu_{\hX^{\pi_n}}= \|f-f_\ell\|_{L^p(\Lambda_T^{\pi_n})}^p< \epsilon^p.
  \end{equation*}
\end{proof}

\begin{proposition}\label{prop: approx Gaussian functionals}
  Let $p\ge 1$, $\varrho\in[1,2)$ and $\alpha\in (\frac{1}{4},\frac{1}{2\varrho})$. Let $X$ be a $d$-dimensional continuous, centred Gaussian process with independent components and covariance $R_X$ such that there exists a $\varrho\in[1,2)$ and $M<\infty$ such that, for every $i\in\{1,\ldots,d\}$ and all $0\le s\le t\le T$,
  \begin{equation*}
    \|R_{X^i}\|_{\varrho;[s,t]^2}\le M|t-s|^{1/\varrho}.
  \end{equation*}
  Let $X^{\pi_n}$ denote the piecewise linear interpolation of $X$ along $\pi_n$. Set $\hX_t:=(t,X_t)$ and $\hX^{\pi_n}_t:=(t,X^{\pi_n}_t)$, $t\in[0,T]$.
  \begin{enumerate}
    \item[(i)] Let $f\colon C([0,T];\R^{d+1})\to\R$ be a measurable and continuous functional. Assume\\ $f(\hX^{\pi_n})\in L^p(\Omega)$ for all $n\in\N$ and that $(|f(\hX^{\pi_n})|^p)_{n\in\N}$ is uniformly integrable. Then, for every $\epsilon>0$ there exists $n_0\in\N$ such that for all $n\ge n_0$ there exists a linear functional $\boldsymbol\ell\colon T((\R^{d+1}))\to\R$ of the form $\hbbX_T^{\pi_n}\mapsto\sum_{|I|\le N}\ell_I\langle e_I,\hbbX^{\pi_n}_T\rangle$, for some $N\in\N_0$ and $\ell_I\in\R$, such that
    \begin{equation*}
      \E[|f(\hX)-\boldsymbol\ell(\hbbX_T^{\pi_n})|^p]<\epsilon^p.
    \end{equation*}
  
    \item[(ii)] Let $f\colon [0,T]\times C([0,T];\R^{d+1})\to\R $ be a measurable, continuous, non-anticipative functional. Assume $f(t,\hX^{\pi_n})\in L^p([0,T]\times \Omega)$ for all $n\in\N$ and that $(|f(t,\hX^{\pi_n})|^p)_{n\in\N}$ is uniformly integrable. Then, for every $\epsilon>0$ there exists $n_0\in\N$ such that for all $n\ge n_0$ there exists a linear functional $\boldsymbol\ell\colon T((\R^{d+1}))\to\R$ of the form $\hbbX_t^{\pi_n}\mapsto\sum_{|I|\le N}\ell_I\langle e_I,\hbbX^{\pi_n}_t\rangle$, for some $N\in\N_0$ and $\ell_I\in\R$, such that
    \begin{equation*}
      \E\Bigl[\int_0^T \big|f(t,\hX)-\boldsymbol\ell(\hbbX^{\pi_n}_t)\big|^p\dd t\Bigr] < \epsilon^p.
    \end{equation*}
  \end{enumerate}
\end{proposition}

\begin{proof}
  (i) Fix $\epsilon>0$. Since $X$ has continuous paths and $X^{\pi_n}$ is the piecewise linear interpolation on a partition with $|\pi_n|\to0$, we have $\sup_{t\in[0,T]}|\hX_t^{\pi_n}-\hX_t|\to0$ a.s.~as $n\to\infty$. By continuity of $f$, it follows that $f(\hX^{\pi_n})\to f(\hX)$ a.s., hence also in probability. Since $(|f(\hX^{\pi_n})|^p)_{n\in\N}$ is uniformly integrable and $f(\hX^{\pi_n})\to f(\hX)$ in probability, \cite[Proposition~4.12]{Kallenberg2002} implies that there exists $n_0\in\N$ such that, for all $n\ge n_0$,
  \begin{equation*}
    \E\bigl[|f(\hX)-f(\hX^{\pi_n})|^p\bigr]<\Bigl(\frac{\epsilon}{2}\Bigr)^p.
  \end{equation*}
  Fix $n\ge n_0$ and let $\nu_n:=\mu_{\hX^{\pi_n}}$ denote the law of $\hX^{\pi_n}$ on $(\widehat{\mathcal C}^{\pi_n},\mathcal B(\widehat{\mathcal C}^{\pi_n}))$. Then,
  \begin{equation*}
    \int_{\widehat{\mathcal C}^{\pi_n}}|f|^p\dd\mu_{\hX^{\pi_n}}
    =\E[|f(\hX^{\pi_n})|^p]<\infty,
  \end{equation*}  
  which yields that $f\in L^p(\widehat{\mathcal C}^{\pi_n})$.
  By \eqref{eq: integrability verification}, the exponential moment condition is satisfied by the law of $\hX^{\pi_n}$. Therefore, Theorem~\ref{thm:Lpmain} yields that there exists a functional $f_\ell\in\mathcal L$ such that
  \begin{equation*}
    \|f-f_\ell\|_{L^p(\widehat{\mathcal C}^{\pi_n})}^p<\Bigl(\frac \epsilon 2\Bigr)^p.
  \end{equation*}
  Moreover, there exists a linear functional $\boldsymbol\ell$ on the signature of $\hX^{\pi_n}$ such that
  \begin{equation*}
    \E[|f(\hX^{\pi_n})-\boldsymbol\ell(\hbbX_T^{\pi_n})|^p]<\Bigl(\frac \epsilon 2\Bigr)^p,
  \end{equation*}
  since
  \begin{equation*}
    \E[|f(\hX^{\pi_n})-\boldsymbol\ell(\hbbX_T^{\pi_n})|^p]=\int_{\widehat{\mathcal C}^{\pi_n}}|f-f_\ell|^p\dd\mu_{\hX^{\pi_n}}=\|f-f_\ell\|_{L^p(\widehat{\mathcal C}^{\pi_n})}^p<\Bigl(\frac \epsilon 2\Bigr)^p.
  \end{equation*}
  Altogether, using the Minkowski inequality, we obtain
  \begin{align*}
    \E[|f(\hX)-\boldsymbol\ell(\hbbX^{\pi_n}_T)|^p]^{\frac 1 p}
    \le \E[|f(\hX)-f(\hX^{\pi_n})|^p]^{\frac 1 p}+\E[|f(\hX^{\pi_n})-\boldsymbol\ell(\hbbX^{\pi_n}_T)|^p]^{\frac 1 p}
    <\epsilon,
  \end{align*}
  for $n$ large enough. This yields the claim.
  
  \medskip
  (ii) Fix $\epsilon>0$. As above, $\sup_{u\in[0,T]}|\hX_u^{\pi_n}-\hX_u|\to 0$ a.s.~as $n\to\infty$. By continuity of $f$ it follows that $f(\cdot,\hX^{\pi_n})\to f(\cdot,\hX)$ $(\d t\otimes\P)$-a.s., hence in $(\d t\otimes\P)$-measure. Since $(|f(t,\hX^{\pi_n})|^p)_{n\in\N}$ is uniformly integrable and $f(t,\hX^{\pi_n})\to f(t,\hX)$ in $(\d t\otimes\P)$-measure, by \cite[Proposition~4.12]{Kallenberg2002} there exists $n_0\in\N$ such that for all $n\ge n_0$,
  \begin{equation*}
    \E\Bigl[\int_0^T|f(t,\hX)-f(t,\hX^{\pi_n})|^p\dd t\Bigr]<\Bigl(\frac{\epsilon}{2}\Bigr)^p.
  \end{equation*}
  Fix $n\ge n_0$. Let $\nu_n:=(\d t\otimes\mu_{\hX^{\pi_n}})\circ\phi^{-1}$ be the push-forward measure on $(\Lambda_T^{\pi_n},\mathcal B(\Lambda_T^{\pi_n}))$, where $\phi(t,\hX^{\pi_n})=\hX^{\pi_n}_{[0,t]}$ is the quotient map as given in \eqref{eq: quotient map}.

  Since $f$ is non-anticipative, it factors through $\phi$, i.e.~there exists a measurable $F\colon\Lambda_T^{\pi_n}\to\R$ such that $f = F\circ \phi$, then $F(\hX_{[0,t]}^{\pi_n})=f(t,\hX^{\pi_n})$. Then, we have that $F\in L^p(\Lambda_T^{\pi_n};\nu_n)$, since
  \begin{equation*}
    \int_{\Lambda_T^{\pi_n}}|F|^p\dd\nu_n
    =\int_{\widehat{\mathcal C}^{\pi_n}}\int_0^T|(F\circ\phi)(t,\hX^{\pi_n})|^p\dd t\dd\mu_{\hX^{\pi_n}}
    =\E\Bigl[\int_0^T|f(t,\hX^{\pi_n})|^p\dd t\Bigr]<\infty.
  \end{equation*}
  Moreover, by \eqref{eq: integrability verification 2} the exponential moment condition is satisfied for $\nu_n$ and therefore by Theorem~\ref{thm:Lpmain} there exists a functional $f_\ell\in\mathcal L_\Lambda$ such that
  \begin{equation*}
    \|F-f_\ell\|_{L^p(\Lambda_T^{\pi_n})}^p<\Bigl(\frac \epsilon 2\Bigr)^p.
  \end{equation*}
  Thus, there exists a linear functional $\boldsymbol\ell$ on the signature of $\hX^{\pi_n}$ such that
  \begin{equation*}
    \E\Bigl[\int_0^T|f(t,\hX^{\pi_n})-\boldsymbol\ell(\hbbX_t^{\pi_n})|^p\dd t\Bigr]<\Bigl(\frac \epsilon 2\Bigr)^p,
  \end{equation*}
  since
  \begin{align*}
    \E\Bigl[\int_0^T|f(t,\hX^{\pi_n})-\boldsymbol\ell(\hbbX_t^{\pi_n})|^p\dd t\Bigr]&=\int_{\widehat{\mathcal C}^{\pi_n}}\int_0^T|(F\circ\phi-f_\ell\circ\phi)(t,\hX^{\pi_n})|^p\dd t\dd\mu_{\hX^{\pi_n}}\\
    &=\int_{\Lambda_T^{\pi_n}}|F-f_\ell|^p\dd\nu_n\\
    &=\|F-f_\ell\|_{L^p(\Lambda_T^{\pi_n})}^p<\Bigl(\frac \epsilon 2\Bigr)^p.
  \end{align*} 
  Combining with the first step and using Minkowski's inequality gives the claim.
\end{proof}

\begin{example}
  Let $X=W^H$ be a $d$-dimensional fractional Brownian motion with Hurst index $H\in(\frac{1}{4},\frac{1}{2}]$, $\varrho=\frac{1}{2H}$ and let $\alpha\in(\frac{1}{4},H)$. Then by \cite[Proposition~15.5]{Friz2010}, there exists a constant $C=C(H)$ such that for all $0\le s<t\le T$
  \begin{equation*}
    \|R_{X^i}\|_{\frac{1}{2H};[s,t]^2}\le C |t-s|^{\frac{1}{2H}}.
  \end{equation*}
  Thus, $W^H$ satisfies the covariance regularity assumption imposed throughout this section. Consequently, all preceding corollaries apply to the piecewise linear interpolations $(W^H)^{\pi_n}$. In particular, for $H=\frac{1}{2}$, fractional Brownian motion coincides with standard Brownian motion. In this case $\varrho=1$ and the admissible range becomes $\alpha\in(\frac{1}{3},\frac{1}{2})$. Hence, the preceding results hold, in particular, for standard $d$-dimensional Brownian motion.
\end{example}

\subsection{Convergence rates for piecewise linear approximations of Gaussian signatures}\label{subsec: approximation error}

In this subsection we derive quantitative convergence estimates for the signatures of piecewise linear approximations of Gaussian processes.

\medskip

We first briefly recall the construction of the canonical rough path lift of a Gaussian process.
Let $X=(X^1,\ldots,X^d)$ be a continuous, centred Gaussian process with independent components. We assume that there exists $\varrho\in[1,2)$ and $M<\infty$ such that 
\begin{equation*}
  \|R_{X^i}\|_{\varrho;[s,t]^2}\le M|t-s|^{\frac{1}{\varrho}}
\end{equation*}
for every $i=1,\ldots,d$ and $0\le s<t\le T$. For a partition $\pi$ of $[0,T]$, let $X^\pi$ denote the piecewise linear interpolation of $X$ along $\pi$. Since $X^\pi$ has bounded variation, its iterated integrals are defined in the classical Riemann--Stieltjes sense. By \cite[Theorem~15.33]{Friz2010}, for every $\alpha\in(\frac{1}{4},\frac{1}{2\varrho})$, the lifted piecewise linear approximations converge in $L^r(\Omega)$, for every $r\ge1$, with respect to the $\alpha$-H\"older rough path metric to a canonical geometric Gaussian rough path. Since $\varrho<2$, one may choose $\alpha>\frac{1}{4}$, so that the first three levels are sufficient to describe the corresponding rough path. Following \cite[Definition~15.34]{Friz2010}, the resulting $G^3(\R^d)$-valued process $\bX$ is called the enhanced Gaussian process associated with $X$, and its sample paths are called Gaussian rough paths.

The number of required levels depends more precisely on $\varrho$. If $\varrho\in[1,\frac{3}{2})$, one may choose $\alpha\in(\frac{1}{3},\frac{1}{2\varrho})$, so that a $G^2(\R^d)$-valued lift is sufficient. If $\varrho\in[\frac{3}{2},2)$, one chooses $\alpha\in(\frac{1}{4},\frac{1}{2\varrho})$, and a $G^3(\R^d)$-valued lift is required. 

Moreover, we denote by $\hbX$ the time-extended Gaussian rough path and by $\hbbX$ its associated signature, which is defined as the unique Lyons extension of $\hbX$.

\begin{lemma}
  Let $p\ge 1$, $\varrho\in[1,2)$ and $\alpha\in(\frac{1}{4},\frac{1}{2\varrho})$. Let $X$ be as in Proposition~\ref{prop: approx Gaussian functionals}, $\hX=(\cdot,X)$ be the time-extended Gaussian process and $\hbbX$ be the corresponding time-extended Gaussian signature. Let $\hbbX^{\pi_n}$ be the signature of the piecewise linear interpolated Gaussian process $\hX^{\pi_n}\in\widehat{\mathcal C}^{\pi_n}$ along $\pi_n$. Let $N\in\N_0$ and let $\boldsymbol\ell=\sum_{|I|\le N}\ell_Ie_I$, $\ell_I\in\R$ be fixed. Then, for every $0<\eta<\frac{1}{\varrho}-\frac{1}{2}$, there exists a constant $C=C(d,N,p,T,M,\varrho,\eta,\boldsymbol\ell)>0$ such that for every $n\in\N$,
  \begin{equation*}
    \sup_{t\in[0,T]}\|\boldsymbol\ell(\hbbX_t)-\boldsymbol\ell(\hbbX_t^{\pi_n})\|_{L^p(\Omega)}\le C |\pi_n|^\eta.
  \end{equation*}
\end{lemma}

\begin{proof}
  If $N=0$, then $\boldsymbol\ell=\ell_\emptyset e_\emptyset$ and since $\langle e_\emptyset,\hbbX_t^{\pi_n}\rangle= \langle e_\emptyset,\hbbX_t\rangle$, the estimate is immediate. We therefore assume that $N>0$.

  Fix $0<\eta<\frac{1}{\varrho}-\frac{1}{2}$ and define $\gamma>\varrho$ by $\frac{1}{\gamma}:=\frac{1}{\varrho}-2\eta$. Then $\frac{1}{\gamma}+\frac{1}{\varrho}>1$. Let $q>2\gamma$.

  For $\mathbf x,\mathbf y\in C([0,T];G^N(\R^d))$, we define
  \begin{equation*}
    \rho^N_{q\textup{-var}}(\mathbf x,\mathbf y):=\max_{k=1,\ldots,N}\rho_{q\textup{-var}}^{(k)}(\mathbf x,\mathbf y):=\max_{k=1\ldots,N}\sup_{(t_i)\subset [0,T]}\Bigl(\sum_{i} |\mathbf x_{t_i,t_{i+1}}^{(k)}-\mathbf y_{t_i,t_{i+1}}^{(k)}|^{q/k}\Bigr)^{k/q}.
  \end{equation*}
  By \cite[Chapter~15.2.3]{Friz2010}, the piecewise linear approximation associated with $(\pi_n)_{n\in\N}$ satisfies the assumptions of \cite[Theorem~2.4.1]{Riedel2013}. Hence, applying \cite[Theorem~2.4.1]{Riedel2013} at truncation level $N$, there exists a constant $C=C(d,q,p,\varrho,\gamma,M,N,T)>0$ such that
  \begin{equation}\label{eq: q-var estimate}
    \|\rho^N_{q\textup{-var}}(\hbbX^{\pi_n,N},\hbbX^N)\|_{L^p(\Omega)}\le C \sup_{t\in[0,T]}\|X_t^{\pi_n}-X_t\|_{L^2(\Omega)}^{1-\frac{\varrho}{\gamma}}.
  \end{equation}
  The theorem is originally formulated for the spatial Gaussian process. Since time extension corresponds to adjoining the deterministic bounded-variation path $t\mapsto t$, the same estimate holds for the time-extended lifts after possibly enlarging the constant $C$. The dependence on $T$ also absorbs the deterministic rescaling from $[0,1]$ to $[0,T]$.

  Now observe that for $t\in[t_j^n,t_{j+1}^n]$ we have with $\lambda=\frac{t-t_j^n}{t_{j+1}^n-t_j^n}$
  \begin{equation*}
    X_t-X_t^{\pi_n}=(1-\lambda)(X_t-X_{t_j^n})+\lambda(X_t-X_{t_{j+1}^n}).
  \end{equation*}
  And in turn, we obtain
  \begin{align*}
    \|X_t-X_t^{\pi_n}\|_{L^2(\Omega)}&\le (1-\lambda)\|X_t-X_{t_j^n}\|_{L^2(\Omega)}+\lambda\|X_t-X_{t_{j+1}^n}\|_{L^2(\Omega)}\\
    &\le \sqrt{dM}|\pi_n|^{\frac{1}{2\varrho}},
  \end{align*}
  where we used that for each component $i=1,\ldots,d$, the covariance assumption gives
  \begin{equation*}
    \|X_t^i-X_s^i\|_{L^2(\Omega)}^2\le \|R_{X^i}\|_{\varrho;[s,t]^2}\le M|t-s|^{1/\varrho}.
  \end{equation*}
  In particular,
  \begin{equation*}
    \sup_{t\in[0,T]}\|X_t^{\pi_n}-X_t\|_{L^2(\Omega)}\le \sqrt{dM}|\pi_n|^{\frac{1}{2\varrho}}.
  \end{equation*}
  Substituting this into \eqref{eq: q-var estimate} yields,
  \begin{equation*}
    \|\rho^N_{q\textup{-var}}(\hbbX^{\pi_n,N},\hbbX^N)\|_{L^p(\Omega)}\le C (dM)^{\frac{1}{2}(1-\frac{\varrho}{\gamma})}|\pi_n|^{\frac{1}{2\varrho}(1-\frac{\varrho}{\gamma})}=C_1|\pi_n|^{\eta},
  \end{equation*}
  where $C_1:=C(dM)^{\frac{1}{2}(1-\frac{\varrho}{\gamma})}$. Let now $I\in\{0,\ldots,d\}^k$ be a multi-index of length $1\le k\le N$. We observe that
  \begin{equation*}
    |\hbbX_t^{\pi_n,(k)}-\hbbX_t^{(k)}|\le \rho_{q\textup{-var}}^{(k)}(\hbbX^{\pi_n,N},\hbbX^N)\le \rho^N_{q\textup{-var}}(\hbbX^{\pi_n,N},\hbbX^N),
  \end{equation*}
  and therefore
  \begin{align*}
    \|\langle e_I,\hbbX_t^{\pi_n}\rangle-\langle e_I,\hbbX_t\rangle\|_{L^p(\Omega)}&\le \|\hbbX_t^{\pi_n,(k)}-\hbbX_t^{(k)}\|_{L^p(\Omega)}\\
    &\le C_1|\pi_n|^{\eta}.
  \end{align*}

  Finally, using Minkowski's inequality, we have
  \begin{align*}
    \sup_{t\in[0,T]}\| \boldsymbol\ell(\hbbX_t)- \boldsymbol\ell(\hbbX^{\pi_n}_t)\|_{L^p(\Omega)}
    &= \sup_{t\in[0,T]}\Bigl\|\sum_{|I|\le N} \ell_I\bigl(\langle e_I,\hbbX_t\rangle-\langle e_I,\hbbX^{\pi_n}_t\rangle\bigr)\Bigl\|_{L^p(\Omega)} \\
    &\le  \sum_{0<|I|\le N}|\ell_I|
    \sup_{t\in[0,T]}\Bigl\|\langle e_I,\hbbX_t\rangle-\langle e_I,\hbbX^{\pi_n}_t\rangle\Bigr\|_{L^p(\Omega)}\\
    &\le C_1 \sum_{0<|I|\le N}|\ell_I| |\pi_n|^{\eta}.
  \end{align*}
  Now set $C_\ell:=C_1\sum_{0<|I|\le N} |\ell_I|$, then $\sup_{t\in[0,T]}\| \boldsymbol\ell(\hbbX_t)- \boldsymbol\ell(\hbbX^{\pi_n}_t)\|_{L^p(\Omega)}\le C_\ell |\pi_n|^\eta$. Since $\gamma$ is determined by $\varrho$ and $\eta$ and $q>2\gamma$ may be fixed once and for all as a function of $\gamma$ the constant $C_\ell$ depends only on $d,N,p,T,M,\varrho,\eta$ and $\boldsymbol\ell$ as asserted.
\end{proof}

Although this result applies in particular to Brownian motion (with $\varrho=1$), in the Brownian setting we can obtain the sharp convergence rate $1/2$, as shown in the following lemma.

Observe that in this case, the canonical Gaussian rough path introduced above coincides with the Stratonovich lift of Brownian motion and the  Brownian signature is given explicitly by iterated Stratonovich integrals. More precisely, let $W$ be a $d$-dimensional Brownian motion. Its Stratonovich lift to a geometric rough path is given by
\begin{equation*}
  \bW_t=\Bigl(1,W_t,\int_0^t W_s\otimes \circ \d W_s\Bigr),\quad t\in[0,T],
\end{equation*}
where the stochastic integral $\int_0^t W_s\otimes\circ \dd W_s$ is defined as a classical Stratonovich integral. It is well known that $\bW_t$ takes values in $G^2(\R^d)$ and is almost surely a geometric $\alpha$-H\"older rough path for any $\alpha\in(\frac 1 3,\frac 1 2)$. We then denote by $\hbW$ the time-extended Stratonovich-enhanced Brownian rough path and by $\hbbW$ its associated signature, which coincides with iterated Stratonovich integrals. We refer to $\hbW$ and $\hbbW$ as the (time-extended) Brownian rough path and (time-extended) Brownian signature, respectively.

\begin{lemma}\label{lem: brownian signature}
  Let $\alpha\in(\frac 1 3,\frac 1 2)$ and $p\ge 1$. Let $W$ be a $d$-dimensional Brownian motion, $\hW=(\cdot,W)$ be the time-extended Brownian motion and $\hbbW$ be the corresponding time-extended Brownian signature. Let $\hbbW^{\pi_n}$ be the signature of the piecewise linear interpolated Brownian motion $\hW^{\pi_n}\in\widehat{\mathcal C}^{\pi_n}$ along $\pi_n$. Let $N\in\N_0$ and let $\boldsymbol\ell=\sum_{|I|\le N}\ell_Ie_I,\ell_I\in\R$ be fixed. Then, there exists a constant $C=C_{d,N,p,T,\boldsymbol\ell}>0$ such that for every $n\in\N$,
  \begin{equation*}
  \sup_{t\in[0,T]}\| \boldsymbol\ell(\hbbW_t)- \boldsymbol\ell(\hbbW^{\pi_n}_t)\|_{L^p(\Omega)}\le C|\pi_n|^{\frac{1}{2}}.
  \end{equation*}
  Consequently,
  \begin{equation}\label{eq: Lp rate}
    \E\Bigl[\int_0^T|\boldsymbol\ell(\hbbW_t)-\boldsymbol\ell(\hbbW_t^{\pi_n})|^p\dd t\Bigr]^{\frac{1}{p}}\le T^{\frac{1}{p}}C|\pi_n|^{\frac{1}{2}}.
  \end{equation}
  In particular, 
  \begin{equation*}
    \E\Bigl[\int_0^T|\boldsymbol\ell(\hbbW_t)-\boldsymbol\ell(\hbbW_t^{\pi_n})|^p\dd t\Bigr]^{\frac{1}{p}}\to 0,
  \end{equation*}
  as $n\to\infty$.
\end{lemma}

\begin{proof}
  Let $I\in\{0,\ldots,d\}^k$ be a multi-index of length $k\ge 0$. First note that for $k=0$, we have $\langle e_I,\hbbW_t\rangle=\langle e_I,\hbbW_t^{\pi_n}\rangle$, and therefore, $\|\langle e_I,\hbbW_t\rangle-\langle e_I,\hbbW^{\pi_n}_t\rangle\|_{L^p(\Omega)}=0$. 
  Now for $k>0$, using \cite[Theorem~1.1.7]{Riedel2013} in the second step, we have for some constant $C_{k,p,T}>0$,
  \begin{equation*}
    \|\langle e_I,\hbbW_t\rangle-\langle e_I,\hbbW^{\pi_n}_t\rangle\|_{L^p(\Omega)}\nonumber\\
    \le \| \hbbW_t^{(k)}-\hbbW_t^{(k),\pi_n}\|_{L^p(\Omega)}\\ 
    \le C_{k,p,T}|\pi_n|^{\frac{1}{2}}.
 \end{equation*}
  
  Thus, using Minkowski's inequality, we have for $t\in[0,T]$
  \begin{align*}
    \| \boldsymbol\ell(\hbbW_t)- \boldsymbol\ell(\hbbW^{\pi_n}_t)\|_{L^p(\Omega)}
    &= \Bigl\|\sum_{|I|\le N} \ell_I\bigl(\langle e_I,\hbbW_t\rangle-\langle e_I,\hbbW^{\pi_n}_t\rangle\bigr)\Bigl\|_{L^p(\Omega)} \\
    &\le  \sum_{0<|I|\le N}|\ell_I|
    \Bigl\|\langle e_I,\hbbW_t\rangle-\langle e_I,\hbbW^{\pi_n}_t\rangle\Bigr\|_{L^p(\Omega)}\\
    &\le  \sum_{0<|I|\le N}|\ell_I| C_{|I|,p,T}|\pi_n|^{\frac{1}{2}}.
  \end{align*}
  Define $C=C_{d,N,p,T,\boldsymbol\ell}:=\sum_{0<|I|\le N}|\ell_I|C_{|I|,p,T}$. Then,
  \begin{equation*}
  \sup_{t\in[0,T]}\| \boldsymbol\ell(\hbbW_t)- \boldsymbol\ell(\hbbW^{\pi_n}_t)\|_{L^p(\Omega)}\le C|\pi_n|^{\frac{1}{2}}.
  \end{equation*}
 
  Now by Fubini's theorem, we obtain
  \begin{equation*}
    \E\Bigl[\int_0^T|\boldsymbol\ell(\hbbW_t)-\boldsymbol\ell(\hbbW^{\pi_n}_t)|^p\dd t\Bigr]
    = \int_0^T\|\boldsymbol\ell(\hbbW_t)-\boldsymbol\ell(\hbbW^{\pi_n}_t)\|^p_{L^p(\Omega)}\dd t\le \int_0^T C^p|\pi_n|^{\frac{p}{2}}\dd t
    = T C^p |\pi_n|^{\frac{p}{2}}.
  \end{equation*}
  Taking the $p$-th root gives \eqref{eq: Lp rate} and as $n\to\infty$ we have that $|\pi_n|\to 0$ and therefore the last claim follows.
\end{proof}

\begin{remark}
  Our result establishes that linear functionals of the Brownian signature can be approximated by the same linear functionals evaluated on the signatures of linearly interpolated Brownian paths. Together with Proposition~4.2 in \cite{Ceylan2025}, this implies that $L^p$-functionals on Brownian rough paths can be approximated by linear combinations of discrete-time signatures.
\end{remark}

\section{Approximation of random ordinary and stochastic differential equations}\label{sec: approximation of SDEs}

In this section we show that we can approximate solutions to random ordinary differentiable equations (ODEs) and stochastic differential equations (SDEs) by linear functionals acting on the time-extended signature of piecewise linearly interpolated Brownian motion.

\subsection{Approximation of random ODEs}

Fix $n\in\N$ and consider the random ODE
\begin{equation}\label{eq: ode}
  Y^{\pi_n}_t = y_0 + \int_0^t \mu(s,Y^{\pi_n}_s) \dd s + \int_0^t \sigma(s,Y^{\pi_n}_s) \dd W^{\pi_n}_s, \quad t \in [0,T],
\end{equation}
where  $y_0 \in \R^m$, $\mu\colon [0,T]\times \R^m \to \R^m$ and $\sigma\colon [0,T] \times \R^m \to \R^{m\times d}$ are continuous functions and $W^{\pi_n}$ is the piecewise linear interpolation along $\pi_n$.

\begin{theorem}
  Let $1\le p<\infty$. Fix $n\in\N$. Suppose that $\mu,\sigma$ satisfy 
  \begin{equation*}
    |\mu(t,x)|+|\sigma(t,x)|\le C(1+|x|),\quad x\in\R^m,
  \end{equation*}
  for some constant $C>0$ and are globally Lipschitz continuous.  Then, for every $\epsilon>0$ there exists
  a linear functional $\boldsymbol\ell\colon T((\R^{d+1}))\to\R^m$ of the form $\hbbW_t^{\pi_n}\mapsto \sum_{|I|\le N}\ell_I\langle e_I,\hbbW^{\pi_n}_t\rangle$, for some $N\in\N_0$ and $\ell_I\in\R^m$, such that
  \begin{equation*}
    \E\Bigl[\int_0^T|Y_t^{\pi_n}-\boldsymbol\ell(\hbbW^{\pi_n}_t)|^p\,dt\Bigr]<\epsilon^p,
  \end{equation*}
\end{theorem}

\begin{proof}
  For notational simplicity assume $m=1$.

  We rewrite the ODE~\eqref{eq: ode} as an ODE driven by the time-extended linear interpolated Brownian motion, by considering $\widehat{\sigma}\colon [0,T]\times\R\to\R^{1\times(d+1)}$ with
  \begin{equation*}
    \widehat{\sigma}=
    \begin{pmatrix}
	  \mu_1, & \sigma_{1},& \cdots &, \sigma_{d}
      \end{pmatrix},
  \end{equation*}
  then 
  \begin{equation*}
    \dd Y_t^{\pi_n}=\widehat\sigma(t,Y_t^{\pi_n})\dd\hW_t^{\pi_n}.
  \end{equation*}
  Let $\Phi(\hW^{\pi_n}_{[0,t]}):=Y_t^{\pi_n}$ denote the solution map on stopped paths. Let $\phi(t,\hW^{\pi_n})=\hW^{\pi_n}_{[0,t]}$ be the quotient map as defined in \eqref{eq: quotient map} and define the finite Borel measure $\nu_n:=(\d t\otimes\mu_{\hW^{\pi_n}})\circ\phi^{-1}$ on $(\Lambda_T^{\pi_n},\mathcal B(\Lambda_T^{\pi_n}))$. Then, by definition of $\nu_n$,
  \begin{align*}
    \int_{\Lambda_T^{\pi_n}}|\Phi|^p\dd\nu_n&=\int_{\widehat{\mathcal C}^{\pi_n}}\int_0^T|(\Phi\circ\phi)(t,\hW^{\pi_n})|^p\dd t\dd\mu_{\hW^{\pi_n}}\\
    &=\E\Bigl[\int_0^T|\Phi(\hW^{\pi_n}_{[0,t]})|^p\dd t\Bigr]\\
    &=\E\Bigl[\int_0^T|Y_t^{\pi_n}|^p\dd t\Bigr]\\
    &\le T \E \Bigl[\sup_{t\in[0,T]}|Y_t^{\pi_n}|^p\Bigr]<\infty,
  \end{align*}
  where we used that $\E[\sup_{t\in[0,T]}|Y_t^{\pi_n}|^p]<\infty$, which follows from global Lipschitz continuity and linear growth of $\mu,\sigma$. Hence $\Phi\in L^p(\Lambda_T^{\pi_n})$.

  Since the exponential moment condition required in Theorem~\ref{thm:Lpmain}~\textup{(ii)} holds on $(\Lambda_T^{\pi_n},\nu_n)$ for linear interpolated Brownian motions, see \eqref{eq: integrability verification 2}, Theorem~\ref{thm:Lpmain}~\textup{(ii)} yields $f_\ell\in\mathcal L_\Lambda$ such that
  \begin{equation*}
    \|\Phi-f_\ell\|_{L^p(\Lambda_T^{\pi_n})}<\epsilon.
  \end{equation*}
  This ensures that the solution $Y^{\pi_n}$ can be approximated by a linear combination of the signature of $\hW^{\pi_n}$, since
  \begin{align*}
    \E\Bigl[\int_0^T|Y_t^{\pi_n}-\boldsymbol\ell(\hbbW^{\pi_n}_t)|^p\dd t\Bigr]
    &=\int_{\widehat{\mathcal C}^{\pi_n}}\int_0^T|(\Phi\circ\phi-f_\ell\circ \phi)(t,\hW^{\pi_n})|^p\dd t\dd\mu_{\hW^{\pi_n}}\\
    &=\int_{\Lambda_T^{\pi_n}}|\Phi-f_\ell|^p\dd\nu_n\\
    &=\|\Phi-f_\ell\|^p_{L^p(\Lambda_T^{\pi_n})}<\epsilon^p.
  \end{align*}
  This concludes the proof.
\end{proof}

\subsection{Approximation of SDEs}

To approximate solutions of SDEs by linear functionals of signatures of piecewise linear interpolated Brownian paths, we use the stability result from Lemma~\ref{lem: brownian signature} and combine it with a global universal approximation result for Brownian signatures from \cite{Ceylan2025}.

\begin{corollary}
  Let $1\le p<\infty$. Consider the stochastic differential equation
  \begin{equation}\label{eq: sde}
    Y_t = y_0 + \int_0^t \mu(s,Y_s) \dd s + \int_0^t \sigma(s,Y_s) \dd W_s, \quad t \in [0,T],
  \end{equation}
  where $y_0 \in \R^m$, $\mu\colon [0,T]\times \R^m \to \R^m$ and $\sigma\colon [0,T] \times \R^m \to \R^{m\times d}$ are continuous functions, and $\int_0^t \sigma(s,Y_s) \dd W_s$ is defined as an It{\^o} integral. Suppose there exists a unique (strong) solution $Y$ to the SDE~\eqref{eq: sde} and that $\mu, \sigma$ satisfy the linear growth condition
  \begin{equation*}
    |\mu(t,x)|+|\sigma(t,x)|\leq C (1 + |x|),\quad x\in\R^m,
  \end{equation*}
  for some constant $C>0$.
 
  Then, for every $\epsilon>0$  there exists a linear function $\boldsymbol\ell\colon T((\R^{d+1}))\to \R^m$ of the form $\hbbW^{\pi_n}_t\mapsto\boldsymbol\ell(\hbbW^{\pi_n}_t):=\sum_{|I|\le N}\ell_I\langle e_I,\hbbW^{\pi_n}_t\rangle$, for some $N\in\N_0$ and $\ell_I\in\R^m$, and $n_0\in\N$ such that for all $n\ge n_0$, 
  \begin{equation*}
    \E\Bigl[\int_0^T|Y_t-\boldsymbol\ell(\hbbW^{\pi_n}_t)|^p\dd t\Bigr]<\epsilon^p.
  \end{equation*}
\end{corollary}

\begin{proof}
  By Proposition~4.6 in \cite{Ceylan2025} we know that we can approximate solutions to SDEs driven by Brownian motions by some linear functional of the Brownian signature, i.e., for every $\epsilon>0$ there exists some linear functional $\boldsymbol\ell$, such that
  \begin{equation*}
    \E\Bigl[\int_0^T|Y_t-\boldsymbol\ell(\hbbW_t)|^p\dd t\Bigr]^{\frac{1}{p}}<\epsilon.
  \end{equation*}
  Moreover, by Lemma~\ref{lem: brownian signature}, for this fixed $\boldsymbol\ell$ we have
  \begin{equation*}
    \E\Bigl[\int_0^T|\boldsymbol\ell(\hbbW_t)-\boldsymbol\ell(\hbbW^{\pi_n}_t)|^p\dd t\Bigr]^{1/p}\to 0,
    \qquad n\to\infty,
  \end{equation*}
  so there exists $n_0\in\N$ such that for all $n\ge n_0$,
  \begin{equation*}
    \E\Bigl[\int_0^T|\boldsymbol\ell(\hbbW_t)-\boldsymbol\ell(\hbbW^{\pi_n}_t)|^p\dd t\Bigr]^{1/p}<\epsilon/2.
  \end{equation*}
  Therefore, by Minkowski’s inequality for all $n\ge n_0$,
  \begin{equation*}
    \E\Bigl[\int_0^T|Y_t-\boldsymbol\ell(\hbbW^{\pi_n}_t)|^p\dd t\Bigr]^{1/p}
    \le \E\Bigl[\int_0^T|Y_t-\boldsymbol\ell(\hbbW_t)|^p\dd t\Bigr]^{1/p} +
    \E\Bigl[\int_0^T|\boldsymbol\ell(\hbbW_t)-\boldsymbol\ell(\hbbW^{\pi_n}_t)|^p\dd t\Bigr]^{1/p}
    <\epsilon,
  \end{equation*}
  which implies the claim.
\end{proof}

\bibliography{quellen}{}
\bibliographystyle{amsalpha}

\end{document}